\documentclass[12pt]{article}

\usepackage{a4wide,epsfig}
\usepackage{amsfonts,amsmath,mathtools,amssymb}
\usepackage[utf8]{inputenc}   
\usepackage[T1]{fontenc}
\usepackage{float}
\usepackage{booktabs}
\usepackage{xcolor}
\usepackage{subcaption}
\usepackage{graphicx}
\usepackage[subpreambles=true]{standalone}

\definecolor{orange}{rgb}{1,0.4,0}
\definecolor{green}{rgb}{0,0.4,0}

\newtheorem{theorem}{Theorem}[section]

\newtheorem{lemma}[theorem]{Lemma}

\newtheorem{remark}{Remark}[section]
\newtheorem{assumption}[theorem]{Assumption}
\newcommand{\proof} [1]
   { \noindent {\bf Proof.} #1 \hfill\rule{0.5em}{1.2ex} \par\medskip}

\newcommand{\calS}{\mathcal{S}}
\newcommand{\R}{\mathbb{R}}

\newcommand{\Dual}[1]{\left[#1 \right]^\prime}

\newcommand{\abs}[1]{\left\lvert{#1}\right\rvert}

\newcommand{\norm}[1]{{\left\lVert{#1}\right\rVert}}

\newcommand{\skpr}[1]{\langle #1 \rangle}

\newcommand{\dual}[2]{\left\langle #1, #2 \right\rangle}
\newcommand{\inner}[2]{\left( #1\,,\,#2\right)}

\newcommand{\OB}{\operatorname{\mathsf{B}}}
\newcommand{\OR}{\operatorname{\mathsf{R}}}
\newcommand{\OP}{\operatorname{\mathsf{P}}}
\newcommand{\OV}{\operatorname{\mathsf{V}}}
\newcommand{\OK}{\operatorname{\mathsf{K}}}
\newcommand{\OI}{\operatorname{\mathsf{I}}}

\newcommand{\OQ}{\operatorname{\mathsf{Q}}}
\newcommand{\OS}{\operatorname{\mathsf{S}}}

\newcommand{\OG}{\operatorname{\mathsf{G}}}

\newcommand{\bb}{\mathsf{b}}

\newcommand{\OOM}{\operatorname{\boldsymbol{\mathsf{M}}}}
\newcommand{\OOB}{\operatorname{\boldsymbol{\mathsf{B}}}}
\newcommand{\OOA}{\operatorname{\boldsymbol{\mathsf{A}}}}

\newcommand{\npi}[1]{}

\newcommand{\whh}{\widehat h}

\numberwithin{equation}{section} 

\begin{document}

\setcounter{page}{1}

\title{Stability analysis of boundary element methods for the wave equation: 
The weakly singular operator in $1d$}
\author{Daniel~Hoonhout$^1$, Richard~L\"oscher$^2$, \\[1mm]
  Carolina~Urz\'ua--Torres$^1$\thanks{The research leading to this
publication received funding from the Dutch Research
Council (NWO) under the NWO-Talent Programme Veni with the project number
VI.Veni.212.253.}}
\date{\small $^1$Delft Institute of Applied Mathematics, TU Delft,
  Mekelweg 4, 2628CD Delft, The Netherlands \\[1mm]
  $^2$Institut f\"ur Angewandte Mathematik, TU Graz, Steyrergasse 30, 8010 Graz, Austria}

\maketitle

\section{Introduction}
Recent work on time-domain boundary integral equations for the wave equation has 
shown that their variational formulations are continuous and inf-sup stable 
\cite{SUT22}, and only occasionally elliptic \cite{ADG09, SUZ21}. Therefore, the discrete
stability of their discretization via space-time boundary element methods (BEM) 
relies on finding pairs of discrete test and trial spaces such that the discrete 
inf-sup condition holds independently of mesh refinement.

In this paper, we use the framework of optimal test spaces \cite{DemkowiczGopalakrishnan:2011}
in order to show the stability and instability of pairs of discrete test and
trial spaces when considering the variational formulation related to the weakly
singular operator $\OV$ arising from the wave equation in $1d$. In particular, 
we focus on the classical choices of low order Lagrangian boundary elements. 
It is worth mentioning that a similar analysis can be carried out for the other 
related boundary integral operators. We refer the reader to \cite{HU:2025} for
recent developments regarding the double layer operator $\OK$.

With this purpose in mind, the structure of this paper is as follows: We first 
revisit the optimal test space framework in an abstract setting and recall 
the tools that we will use in our analysis. In Section~\ref{sec:Vin1d}, we 
briefly summarize the properties of $\OV$ for one spatial dimension. Then we 
combine these two ingredients in Section~\ref{ssec:optTestV} to characterize the
related optimal test spaces, such that in Section~\ref{sec:mainResults} we 
present our main stability results, subject to a constraint on the mesh.
Section~\ref{sec:numerics} is devoted to presenting numerical examples to
explore the sharpness of our predictions. Finally, in Section 5, we will
extend the analysis to show stability even without the mesh constraint.

\section{Abstract Setting}\label{sec:abstract-setting}
Let us consider the spaces $X$, $Y$ and their dual spaces $X'$ and $Y'$. For
given $f\in Y'$, we want to find the solution $u\in
X$ of the operator equation 
\begin{equation}\label{eq:operator-equation}
  \OB u = f \quad \text{ in }Y',
\end{equation}
where $\OB:X\to Y'$ is assumed to be a bounded linear operator, i.e., there
exists $C^B>0$ such that $\norm{\OB v}_{Y'}\leq C^B \norm{v}_X$ for all $v\in
X$. Note that \eqref{eq:operator-equation} is equivalent to the variational 
formulation: Find $u\in X$ such that
\begin{equation}\label{eq:VF-abstract}
  \bb(u,q):= \skpr{\OB u,q} = \skpr{f,q}\quad \forall q\in Y,
\end{equation}
where $\skpr{\cdot,\cdot}$ denotes the duality pairing as an extension of the $H$-
inner product and $\bb(\cdot,\cdot):X\times Y\to\mathbb{R}$ is the bilinear form
induced by the operator $\OB$. In order to ensure unique solvability of
\eqref{eq:VF-abstract}, we assume that $\OB\in \mathcal{L}_{\text{is}}(X,Y')$,
where
\begin{equation*}
 \mathcal{L}_{\text{is}}(X,Y'):=\{\OB:X\to Y' \text{ is an isomorphism}\}.
\end{equation*}
By the well-established theory for Petrov-Galerkin formulations
\cite{BabuskaAziz:1972,Brezzi:1974} this is the case if the following conditions
are met:
\begin{flalign}
\text{\bf(BN1)}&\qquad \exists \alpha_{\text{is}}^B >0 \ \forall v\in X:\,
\alpha_{\text{is}}^B\norm{v}_X\leq \sup_{0\neq q\in Y}\frac{\bb(v,q)}{\norm{
q}_Y},\label{ass:B bounded inverse}\\
\text{\bf(BN2)}&\qquad \forall q\in Y\setminus\{0\}\,\exists v_q\in X:\, \bb(
v_q,q) \neq 0. \label{ass:B surjective}
\end{flalign}
We will use that if $\OB\in \mathcal{L}_{\text{is}}
(X,Y')$, then \eqref{eq:VF-abstract} is equivalent to the minimization problem 
\begin{equation}
  u = \arg\min_{w\in X}\frac{1}{2}\norm{\OB w-f}_{Y'}^2.\label{eq:minProb}
\end{equation}  
Let us denote the inner product of $Y$ by $\inner{\cdot}{\cdot}_Y$.
We realize the dual norm in \eqref{eq:minProb} by considering the Riesz operator
$\OR_Y\in \mathcal{L}_{\text{is}}(Y,Y')$ defined as
\begin{equation*}
  \skpr{\OR_Yp,q} = \inner{p}{q}_Y, \quad \forall p,q\in Y,
\end{equation*}
which is by definition self-adjoint and elliptic and an isometry. Thus, 
we have that
\begin{equation*}
  \norm{f}_{Y'} = \norm{\OR_Y^{-1}f}_Y\quad \forall f\in Y'.
\end{equation*}
Then, the minimization problem becomes: find $u\in X$ as minimizer of 
\begin{equation*}
J(w):=\frac{1}{2}\norm{\OR_Y^{-1}(\OB w-f)}_{Y}^2 = \frac{1}{2}\skpr{\OR_Y^{-1}(\OB w-f),
\OB w-f}.
\end{equation*}
Since the functional $J$ is strictly convex, continuous and coercive, the
unique minimizer $\bar{u}$ satisfies the gradient equation
\begin{equation}\label{eq:gradient equation}
\delta J(\bar{u})(v) = \skpr{\OR_Y^{-1}(\OB\bar{u}-f),\OB v} =0 \quad \forall v\in X,
\end{equation}
which is equivalent to
\begin{equation}\label{eq:VF-LSabstract}
  \skpr{\OB\bar{u},\OR_Y^{-1}\OB v} = \skpr{f,\OR_Y^{-1}\OB v}\quad \forall v\in X.
\end{equation}

This perspective allows us to easily understand the following result.
\begin{lemma}\label{le:unique-solvability-conditions}
For each $f\in Y'$, the variational formulation \eqref{eq:VF-LSabstract} admits
a unique solution $u\in X$, satisfying
  \begin{equation*}
    \norm{u}_X\leq \frac{1}{\alpha_{\text{is}}^B}\norm{f}_{Y'}.
  \end{equation*} 
\end{lemma}
\proof{
Let $\OB^\ast:Y\to X'$ denote the adjoint of $\OB$. The operator
$$\calS:=\OB^\ast \OR_Y^{-1}\OB :X\to X'$$
is by definition self-adjoint. Moreover, since $\OB\in\mathcal{L}_{\text{is}}(X,
Y')$ and $\OR_Y\in\mathcal{L}_{\text{is}}(Y,Y')$, we have that $\calS\in\mathcal{
L}_{\text{is}}(X,X')$, which ensures unique solvability. Using \textbf{(BN1)}
and $\OR_Y^{-1}\OB \in\mathcal{L}_{\text{is}}(X,Y)$, the bound is easily obtained
as follows
\begin{align*}
\alpha_{\text{is}}^B\norm{u}_X &\leq \sup_{0\neq q\in Y}\frac{\bb(u,q)}{\norm{q
}_Y} =
\sup_{0\neq v\in X}\frac{\skpr{\OB u,\OR_Y^{-1}\OB v}}{\norm{\OR_Y^{-1}\OB v}_Y}\\
&= \sup_{0\neq v\in X} \frac{\skpr{f,\OR_Y^{-1}\OB v}}{\norm{\OR_Y^{-1}\OB v}_Y}
= \sup_{0\neq q\in Y}\frac{\skpr{f,q}}{\norm{q}_Y} = \norm{f}_{Y'}.
  \end{align*} 
}

\begin{remark}
As $\OR_Y^{-1}\OB \in \mathcal{L}_{\text{is}}(X,Y)$, each $q\in Y$ can be
represented as $q = \OR_Y^{-1}\OB  v$ for a unique $v\in X$. This immediately
shows that \eqref{eq:VF-abstract} and \eqref{eq:VF-LSabstract} are equivalent.
In particular, using the transformed trial space $Y = \OR_Y^{-1}\OB (X)$ as
test space leads to an elliptic formulation. This will be the key idea guiding
the design of discrete test spaces that lead to stability, as we will see in
the next subsection.
\end{remark}

\subsection{The optimal and the practical test space}
This subsection summarizes an abstract approach for constructing stable pairs of
trial and test spaces, closely related to discontinuous Petrov--Galerkin (DPG)
methods developed by Demkowicz and Gopalakrishnan, see \cite{DeG25} and
references therein. In contrast to the classical DPG setting, the framework here
is restricted to \emph{conforming discrete trial and test spaces}, avoiding the
use of broken spaces. In order to make this
paper self contained, the main steps are outlined and the analysis is adapted
to the present setting and notation.

Let $X_h\subset X$ be a conforming, finite dimensional trial space. Then, the
discrete variational formulation of \eqref{eq:VF-LSabstract} is: Find $u_h\in
X_h$ such that
\begin{equation}\label{eq:DVF-LSabstract-optimal}
  \skpr{\OB u_h,\OR_Y^{-1}\OB v_h} = \skpr{f,\OR_Y^{-1}\OB v_h}\quad \forall
  v_h\in X_h.
\end{equation}
Unique solvability follows by the Lemma of Lax--Milgram as in the continuous 
case, since the operator $\calS\in\mathcal{L}_{\text{is}}(X,X')$ is elliptic. In
particular, when considering the \emph{optimal} test space
\begin{equation*}
  Y_h^{\text{opt}} := \OR_Y^{-1}\OB (X_h)\subset Y,
\end{equation*}
we ensure a unique solution of the discrete variational formulation: Find
$u_h\in X_h$ such that
\begin{equation}\label{eq:VF optimal test spaces}
  \bb(u_h,q_h) = \skpr{f,q_h}\quad \forall q_h\in Y_h^{\text{opt}}.
\end{equation}

Moreover, one can show quasi-optimality, as summarized in the following
Theorem.
\begin{theorem}\label{thm:optimal setting main results}
For any $f\in Y'$ the variational formulation \eqref{eq:VF optimal test spaces}
admits a unique solution $u_h\in X_h$, satisfying
\begin{equation*}
  \norm{u-u_h}_{X} \leq \frac{C^B}{\alpha_{\text{is}}^B}\inf_{v_h\in X_h}
  \norm{u-v_h}_X,
\end{equation*}
where $u\in X$ denotes the unique solution of \eqref{eq:VF-abstract}.
\end{theorem}
\proof{
Let $v_h\in X_h$ be arbitrary but fixed. By definition, there exists a unique
$\bar q_h = \OR_Y^{-1}\OB  v_h \in Y_h^{\text{opt}}$, i.e.,
\begin{equation*}
  \inner{\bar q_h}{q}_Y = \bb(v_h,q)\quad \forall q\in Y.
\end{equation*}
Combining this with the continuous inf-sup stability \textbf{(BN1)}, we get
\begin{equation*}
\alpha_{\text{is}}^B \norm{v_h}_X \leq \sup_{0\neq q\in Y} \frac{\bb(v_h,q)}
{\norm{q}_Y} = \sup_{0\neq q\in Y}\frac{\inner{\bar q_h}{q}_Y}{\norm{q}_Y} =
\sup_{0\neq q\in Y} \frac{\skpr{\OR_Y \bar q_h,q}}{\norm{q}_Y} = \norm{\OR_Y
\bar q_h}_{Y'} = \norm{\bar q_h}_Y.
\end{equation*}
Furthermore, using the $Y$-orthogonal projection $\OP_{Y_h^{\text{opt}}}
:Y\to Y_h^{\text{opt}}$, defined as
\begin{equation*}
\inner{\OP_{Y_h^{\text{opt}}}q}{q_h}_Y = \inner{q}{q_h}_Y\quad \forall q_h\in
Y_h^{\text{opt}},
\end{equation*}
we obtain by the stability of the projection 
\begin{align*}
\norm{\bar q_h}_Y &= \sup_{0\neq q\in Y}\frac{\inner{\bar q_h}{q}_Y}{\norm{q}_Y}\\ &=
\sup_{0\neq q\in Y}\frac{\inner{\bar q_h}{\OP_{Y_h^{\text{opt}}}q}_Y}{\norm{q}_Y}
\leq
\sup_{0\neq q\in Y}\frac{\inner{\bar q_h}{\OP_{Y_h^{\text{opt}}}q}_Y}{\norm{
\OP_{Y_h^{\text{opt}}}q}_Y} \leq \sup_{0\neq q_h\in Y_h^{\text{opt}}} \frac{
\inner{\bar q_h}{q_h}_Y}{\norm{q_h}_Y}
\end{align*}
and thus
\begin{equation*}
  \alpha_{\text{is}}^B \norm{v_h}_X \leq \norm{\bar q_h}_Y \leq  \sup_{0\neq
  q_h\in Y_h^{\text{opt}}}\frac{\bb(v_h,q_h)}{\norm{q_h}_Y}.
\end{equation*}
Therefore, the discrete inf-sup stability holds with the same constant as on
the continuous level. Since $\text{dim}(Y_h^{\text{opt}}) = \text{dim}(X_h),$
this is sufficient to guarantee unique solvability. The quasi best
approximation error estimate follows by standard arguments, see
\cite{XuZikatanov:2003}.
}

Unfortunately, the action of $\OR_Y^{-1}\OB $ cannot generally be computed
explicitly, since the inverse of $\OR_Y$ is usually not explicitly available.
Moreover, the space $Y_h^{\mathrm{opt}}$ need not be a subspace of any readily
constructible discrete space.

In order to compute a \emph{practical} test space, we will thus replace the 
operator $\OR_Y$ by some suitable approximation, as it is done in the framework
of DPG methods \cite{DeG25} or least squares methods \cite{MSS24}. Therefore,
let us first note that, when introducing $p= \OR_Y^{-1}(f-\OB u_h)\in Y$ the
variational formulation \eqref{eq:DVF-LSabstract-optimal} is equivalent to the 
saddle point formulation: Find $(p,u_h)\in Y\times X_h$ such that
\begin{equation}\label{eq:SPF-optimal}
\begin{aligned}
\hphantom{-\bb}\inner{p}{q}_Y &+& \bb(u_h,q) &=& \skpr{f,q}&\quad\forall q\in Y,\\
-\bb(v_h,p) &&             &=& 0         &\quad \forall v_h\in X_h.
\end{aligned}
\end{equation}

Now, choosing a conforming discrete space $Y_h\subset Y$, we consider the fully
discrete variational formulation: Find $(p_h,\tilde u_h)\in Y_h\times X_h$ such
that
\begin{equation}\label{eq:SPF-practical}
\begin{aligned}
\hphantom{-\bb} \inner{p_h}{q_h}_Y &+& \bb(\tilde u_h,q_h) &=& \skpr{f,q_h}&\quad
\forall q_h\in Y_h,\\
-\bb(v_h,p_h) &&           &=& 0         &\quad \forall v_h\in X_h.
\end{aligned}
\end{equation}
Substituting $p\in Y$ by $p_h\in Y_h$ corresponds to replacing $\OR_Y$ by a
discrete counterpart
\begin{equation*}
  \OR_{Y_h} := \OP_{Y_h'}\OR_Y\mathcal{E}_Y: Y_h\to Y_h',
\end{equation*}
where $\mathcal{E}_Y:Y_h\to Y$ denotes the embedding/extension operator, and
$\OP_{Y_h'}: Y'\to Y_h'$ the $Y'$-orthogonal projection operator.

Note that, in this setting, the condition $\dim(X_h)=\dim(Y_h)$ is not required
for unique solvability. This provides additional flexibility in the choice of
the test space $Y_h$.

In order to maintain stability when approximating $p\in Y$ by $p_h\in Y_h$, we
need to verify that the discrete inf-sup condition 
\begin{equation}\label{eq:discrete-inf-sup-B}
\exists c^B_{0}>0 \ \forall v_h\in X_h:\, c^B_{0}\norm{v_h}_{X}\leq \sup_{0\neq 
q_h\in Y_h}\frac{\bb(v_h,q_h)}{\norm{q_h}_Y},
\end{equation}
is still satisfied. We stress that this condition has to be checked independently
of the continuous inf-sup condition \textbf{(BN1)}, which we assumed to hold.
Although it is expected that \eqref{eq:discrete-inf-sup-B} will hold when $Y_h$
approximates $Y$ ``well enough'', one has to establish a uniform lower bound
$c^B_{0}>0$ independently of $\text{dim}(X_h)$ and $\text{dim}(Y_h)$ of the
trial and test spaces. For this purpose, we consider the \emph{practical} test
space
\begin{equation*}
  Y_h^{\text{prac}} = \OR_{Y_h}^{-1}\OB_h(X_h)\subset Y_h,
\end{equation*} 
where $\OB_h = \OP_{Y_h'}\OB\mathcal{E}_X:X_h\to Y_h'$ with the extension
operator $\mathcal{E}_X: X_h\to X$.

We note that if the discrete inf-sup condition \eqref{eq:discrete-inf-sup-B}
holds, then $\text{dim}(Y_h^{\text{prac}}) = \text{dim}(X_h)$ as $\OR_{Y_h}\in
\mathcal{L}_{\text{is}}(Y_h,Y_h')$ and $\O\OB_h:X_h\to Y_h'$ is injective.
Indeed, the following Lemma illustrates the applicability of the practical
test space.

\begin{lemma}
Let $f\in Y'$ and let \eqref{eq:discrete-inf-sup-B} hold true. Then the discrete
saddle point problem \eqref{eq:SPF-practical} admits a unique solution $(p_h,
\tilde{u}_h)\in Y_h\times X_h$. Moreover, \eqref{eq:SPF-practical} is equivalent
to the variational formulation: Find $\tilde{u}_h\in X_h$ such that
\begin{equation}\label{eq:DVF-practical}
\bb(\tilde u_h,\tilde q_h) = \skpr{f,\tilde q_h}\quad \forall \tilde
q_h\in Y_h^{\text{prac}}. 
\end{equation} 
\end{lemma}
\proof{
By definition the practical test space is given by
\begin{equation*}
  Y_h^{\text{prac}} = \{\tilde q_h\in Y_h: \,\exists v_h\in X_h\, \text{ such
  that } \inner{\tilde q_h}{q_h}_Y = \bb(v_h,q_h)\quad \forall q_h\in Y_h \}.
\end{equation*}  
Then from \eqref{eq:SPF-practical}, we compute for any fixed $v_h\in X_h$
\begin{equation*}
  0 = \bb(v_h,p_h) = \inner{\tilde q_h}{p_h}_Y = \inner{p_h}{\tilde q_h}_Y=
  \skpr{f,\tilde q_h} - \bb(\tilde u_h,\tilde q_h),
\end{equation*}
or equivalently
\begin{equation*}
  \bb(\tilde u_h,\tilde q_h) = \skpr{f,\tilde q_h},\quad \forall \tilde q_h
  \in Y_h^{\text{prac}}.
\end{equation*} 
Conversely, if $\tilde{u}_h\in X_h$ is a solution of \eqref{eq:DVF-practical},
define $p_h\in Y_h$ as the unique solution of
\begin{equation*}
  (p_h,q_h)_Y = \skpr{f,q_h} - \bb(\tilde u_h,q_h),\quad \forall q_h\in Y_h.
\end{equation*}
Testing with $\tilde q_h\in Y_h^{\text{prac}}\subset Y_h$ shows that
\begin{equation*}
  (p_h,\tilde q_h)_Y = \skpr{f,\tilde q_h} - \bb(\tilde u_h,\tilde q_h) =0,
\end{equation*}
and again by the definition of the practical test space, for arbitrary $v_h\in
X_h$ we find $\tilde q_h\in Y_h^{\text{prac}}$ such that
\begin{equation*}
  \bb(v_h,p_h) = (\tilde q_h,p_h)_Y = (p_h,\tilde q_h)_Y =0,
\end{equation*}
showing that \eqref{eq:DVF-practical} implies \eqref{eq:SPF-practical}. 
}

\begin{remark}
\label{rmk:Equivalence}
The fact that both formulations are
equivalent allows us to choose \eqref{eq:DVF-practical} for the analysis and
to consider the saddle point system \eqref{eq:SPF-practical} for the practical
implementation.
\end{remark}
Next, we point out the relation between the optimal and the
practical test space.
\begin{lemma}
  If the discrete inf-sup condition \eqref{eq:discrete-inf-sup-B} holds, then  
  \begin{equation*}
    Y_h^{\text{prac}} = \OP_{Y_h}(Y_h^{\text{opt}}),
  \end{equation*}
  where $\OP_{Y_h}:Y\to Y_h$ is the $Y$-orthogonal projection, defined as
  \begin{equation*}
    \inner{\OP_{Y_h}q}{q_h}_Y = \inner{q}{q_h}_Y\quad \forall q_h\in Y_h.
  \end{equation*}
\end{lemma}
\proof{
Let $\tilde q_h\in Y_h^{\text{prac}} = \OR_{Y_h}^{-1}\OB_h(X_h)\subset Y_h$ be
arbitrary but fixed. Then there exists some $v_h\in X_h$ such that 
$\tilde q_h\in Y_h$ is the unique solution of 
\begin{equation*}
\inner{\tilde q_h}{q_h}_Y = \bb(v_h,q_h)\quad \text{ for all }
q_h\in Y_h.
\end{equation*}
Moreover, let $\tilde q\in Y_h^{\text{opt}}$ be the unique solution of  
\begin{equation*}
\inner{\tilde q}{q}_Y = \bb(v_h,q)\quad \text{for all }q\in Y.
\end{equation*}
Since $Y_h\subset Y$, we combine these two equations with the definition of
$\OP_{Y_h}$ to arrive to
\begin{align*}
\inner{\tilde q_h}{q_h}_Y = \bb(v_h,q_h) = \inner{\tilde q}{q_h}_Y =
\inner{\OP_{Y_h}\tilde q}{q_h}_Y
\end{align*}
Hence, $\tilde q_h = \OP_{Y_h}\tilde q$, i.e., $Y_h^{\text{prac}}\subseteq
\OP_{Y_h}(Y_h^{\text{opt}})$. Since $\text{dim}(Y_h^{\text{prac}}) = \text{dim}
(X_h)=\text{dim}(Y_h^{\text{opt}})$, the dimensions are equal and we obtain the
desired equality.    
}

\subsection{Error estimates}
We now proceed to prove the quasi-optimality of the scheme. When subtracting
\eqref{eq:DVF-practical} from \eqref{eq:VF-abstract}, we immediately obtain the
Galerkin orthogonality
\begin{equation*}
\bb(u-\tilde u_h,\tilde q_h) = 0\quad \forall \tilde q_h\in
Y_h^{\text{prac}}\subset Y_h, 
\end{equation*} 
which induces the Galerkin projection $\OG_h:X\to X_h$ defined as $\OG_h u =
\widetilde{u}_h$. Let us first show the following auxiliary result. 

\begin{lemma}
For all $v_h\in X_h$ it holds that   
\begin{equation}\label{eq:supremum-equality}
\sup_{0\neq q_h\in Y_h} \frac{\bb(v_h,q_h)}{\norm{q_h}_Y} = \sup_{0\neq
\tilde q_h\in Y_h^{\text{prac}}} \frac{\bb(v_h,\tilde q_h)}{\norm{\tilde
q_h}_Y}.
\end{equation}
\end{lemma}
\proof{
Let $0\neq v_h\in X_h$ be arbitrary but fixed and $\tilde q_h\in Y_h^{\text{prac}}$
be the unique solution of
\begin{equation*}
\inner{\tilde q_h}{q_h}_Y = \bb(v_h,q_h)\quad \forall q_h\in Y_h.
\end{equation*}
We note that
\begin{equation*}
\norm{p_h}_Y = \sup_{0\neq q_h\in Y_h}\frac{\inner{p_h}{q_h}_Y}{\norm{q_h}_{Y}} =
\frac{\inner{p_h}{p_h}_Y}{\norm{p_h}_{Y}}= \sup_{0\neq q_h\in Y_h^{\text{prac}}}
\frac{\inner{p_h}{q_h}_Y}{\norm{q_h}_{Y}}, \quad \text{for any }0\neq p_h\in
Y_h^{\text{prac}},
\end{equation*}
which allows us to compute
\begin{align*}
\sup_{0\neq q_h\in Y_h}\frac{\bb(v_h,q_h)}{\norm{q_h}_Y} &= \sup_{0\neq q_h
\in Y_h}\frac{\inner{\tilde q_h}{q_h}_Y}{\norm{q_h}_Y}
= \sup_{0\neq q_h\in Y_h^{\text{prac}}}\frac{\inner{\tilde q_h}{q_h}_Y}{\norm{
q_h}_Y}
= \sup_{0\neq q_h\in Y_h^{\text{prac}}}\frac{\bb(v_h,q_h)}{\norm{q_h}_Y}.
\end{align*}
}

Using the discrete inf-sup stability \eqref{eq:discrete-inf-sup-B}, we can
now derive the bound of the Galerkin projection as follows
\begin{align*}
\norm{\OG_hv}_X&\leq \frac{1}{c^B_{0}}\sup_{0\neq q_h\in Y_h}\frac{\bb(\OG_hv,
q_h)}{\norm{q_h}_Y} = \frac{1}{c^B_{0}}\sup_{0\neq \tilde q_h\in Y_h^{\text{
prac}}}\frac{\bb(\OG_hv,\tilde q_h)}{\norm{\tilde q_h}_Y}\\
&= \frac{1}{c^B_{0}}\sup_{0\neq \tilde q_h\in Y_h^{\text{prac}}}\frac{\bb(v,
\tilde q_h)}{\norm{\tilde q_h}_Y} \leq \frac{C^B}{c^B_{0}}\norm{v}_X.
\end{align*}
Since $\norm{\OG_h}_{X\to X} = \norm{\OI-\OG_h}_{X\to X}$, see
\cite[Cor. 2]{Carstensen:2025}, we get 
\begin{align*}
\norm{u-\tilde u_h}_X = \norm{(\OI-\OG_h)(u-v_h)}_X\leq \norm{\OI-\OG_h}_{X\to
X}\norm{u-v_h}_X\leq \frac{C^B}{c^B_{0}}\norm{u-v_h}_X.
\end{align*}

Combining all these pieces, we have the following result. 
\begin{theorem}[{\cite[Thm. 5.2]{DeG25}}]\label{thm:practical-error-estimates}
Let $f\in Y'$ and let the discrete inf-sup condition \eqref{eq:discrete-inf-sup-B}
hold true. Then \eqref{eq:DVF-practical} admits a unique solution $\tilde u_h
\in X_h$ satisfying
\begin{equation*}
\norm{u-\tilde u_h}_X\leq \frac{C^B}{c^B_{0}}\inf_{v_h\in X_h}\norm{u-v_h}_X,
\end{equation*}
where $u\in X$ denotes the unique solution of \eqref{eq:VF-abstract}.  
\end{theorem}
\begin{remark}
When comparing Theorems~\ref{thm:optimal setting main results} and
\ref{thm:practical-error-estimates}, we see the only difference when considering
the practical test space instead of the optimal one is in the constant.
\end{remark}

\subsection{Error estimator}
In the discretized saddle point system \eqref{eq:SPF-practical}, we introduced
the quantity $p_h\in Y_h$, which is a realization of the discrete residual. In
the following, we will show that it can be used as an error estimator if the
discrete inf-sup condition \eqref{eq:discrete-inf-sup-B} is satisfied.

In order to get there, we recall that \eqref{eq:discrete-inf-sup-B} is equivalent
to the existence of a Fortin operator \cite[Thm. 1]{ErnGuermond:2016}, i.e.,
there exists a linear operator $\Pi_h:Y\to Y_h$
satisfying
\newcounter{descitemF}
\renewcommand\thedescitemF{F\arabic{descitemF}}
\begin{description}
\item[(F1)]\refstepcounter{descitemF}\label{ass:fortin 1} $\bb(v_h,q-\Pi_h q
 ) = 0$ for all $v_h \in X_h$, and for all $q\in Y$, 
\item[(F2)]\refstepcounter{descitemF}\label{ass:fortin 2} $\norm{\Pi_h q}_Y \leq
\dfrac{C^B}{c^B_{0}} \norm{q}_Y$ for all $q\in Y$.
\end{description}

Now we are in the position to introduce the following result.
\begin{theorem}[{\cite[Thm. 6.4]{DeG25}}]\label{thm:errorEstimator}
Let $X_h\subset X$ and $Y_h\subset Y$ be such that the discrete inf-sup condition 
\eqref{eq:discrete-inf-sup-B} holds and let $u\in X$ and $(p_h,\tilde{u}_h)
\in Y_h\times X_h$ be the unique solutions of \eqref{eq:VF-abstract} and
\eqref{eq:SPF-practical}, respectively. Then
\begin{equation*}
\frac{1}{C^B}\norm{p_h}_Y \leq \norm{u-\tilde{u}_h}_X\leq \frac{C^B}
{\alpha_{\text{is}}^B c^B_{0}}\norm{p_h}_Y + \text{osc}(f),
\end{equation*} 
where $\text{osc}(f) := \dfrac{1}{\alpha_{\text{is}}^B}\norm{f\circ (\OI-\Pi_h)
}_{Y'}$ denotes the data oscillation.
\end{theorem}
\proof{
For the lower estimate, recall that $\skpr{f,q} = \bb(u,q)$ for all $q\in Y$.
Then, testing the first line in \eqref{eq:SPF-practical} with $q_h=p_h\in Y_h$,
we compute
\begin{align*}
\norm{p_h}_Y^2 &= \inner{p_h}{p_h}_Y = \skpr{f,p_h}-\bb(\tilde u_h,p_h) = \bb(u-
\tilde u_h,p_h) \leq C^B\norm{u-\tilde{u}_h}_{X}\norm{p_h}_Y.
\end{align*}
For the upper bound we use \eqref{ass:B bounded inverse}, the properties of the 
Fortin operator \eqref{ass:fortin 1}-\eqref{ass:fortin 2} and the first line in
\eqref{eq:SPF-practical} to compute 
\begin{align*}
\alpha_{\text{is}}^B \norm{u-\tilde{u}_h}_X & \leq \sup_{0\neq q\in Y}\frac{
\bb(u-\tilde{u}_h,q)}{\norm{q}_Y} \\
& = \sup_{0\neq q\in Y}\left[\frac{\bb(u-\tilde{u}_h,q-\Pi_hq)}{\norm{
q}_Y} + \frac{\bb(u-\tilde{u}_h,\Pi_hq)}{\norm{q}_Y}\right]\\
& \leq \sup_{0\neq q\in Y} \frac{\bb(u,q-\Pi_hq)}{\norm{q}_Y} + \sup_{0\neq
q\in Y}\frac{\bb(u-\tilde{u}_h,\Pi_h q)}{\norm{q}}\\
& = \sup_{0\neq q\in Y}\frac{\skpr{f,(\OI-\Pi_h)q}}{\norm{q}_Y} + \sup_{0\neq q
\in Y} \frac{\inner{p_h}{\Pi_hq}_Y}{\norm{q}_Y}\\
& \leq \norm{f\circ(\OI-\Pi_h)}_{Y'} + \frac{C^B}{c^B_{0}}\norm{p_h}_Y,
\end{align*}     
which concludes the proof. 
}
\begin{remark}
The oscillation term can be bounded by the best approximation in $X_h$, i.e., 
for any $v_h\in X_h$ we compute  
 \begin{align*}
\norm{f\circ(\OI-\Pi_h)}_{Y'} &=\sup_{0\neq q\in Y}\frac{\skpr{f,(\OI-\Pi_h)q}}
{\norm{q}_Y} = \sup_{0\neq q\in Y}\frac{\bb(u-v_h,(\OI-\Pi_h)q)}{\norm{q}_Y}\\
&\leq C^B \norm{u-v_h}_X\norm{\OI-\Pi_h}  = C^B \norm{u-v_h}_X\norm{\Pi_h}\leq
\frac{[C^B]^2}{c^B_{0}}\norm{u-v_h}_X.
 \end{align*}  
\end{remark}

\section{Application to the Weakly Singular Operator for the Wave Equation in 
One Dimension}\label{sec:Vin1d}
We proceed to apply the framework from Section~\ref{sec:abstract-setting}
in order to investigate the stability of some standard discretizations of the
variational formulations for the weakly singular boundary integral equation
for the wave equation in $1d$. We reckon this is not a very standard setting,
so we begin by briefly introducing the framework in subsection~\ref{ssec:PropV}
and then move to building the required pieces in subsection~\ref{ssec:optTestV}.
We conclude this Section by using all these ingredients to show our first
stability results.

\subsection{Properties of the Weakly Singular Operator}\label{ssec:PropV}
For a bounded one-dimensional spatial domain $\Omega = (0,L)$ and a given (finite) 
time horizon $T>0$, we consider the following Dirichlet initial boundary value 
problem:
\begin{equation}\label{DBVP}
  \left. \begin{array}{rcll}
           \partial_{tt} u(x,t) - \partial_{xx} u(x,t)
           & = & 0 & \mbox{for} \; (x,t) \in
                     Q := (0,L) \times (0,T), \\[1mm]
           u(x,0) = \partial_t u(x,t)_{|t=0}
           & = & 0 & \mbox{for} \; x \in (0,L),\\[1mm]
           u(0,t)
           & = & g_0(t) & \mbox{for} \; t \in (0,T), \\[1mm]
           u(L,t)
           & = & g_L(t) & \mbox{for} \; t \in (0,T).
         \end{array} \right \}
\end{equation}
In the one-dimensional case, the fundamental solution of the wave equation is 
the Heaviside function
\begin{equation*}
U^*(x,t) = \frac{1}{2} \, {\mathsf{H}}(t-|x|),
\end{equation*}
and we can represent the solution $u$ of (\ref{DBVP}) by using the single layer 
potential
\begin{equation*}
  u(x,t) := (\widetilde{\OV}z)(x,t) = \frac{1}{2} \int_0^{t-|x|} z(0,s) \, ds +
  \frac{1}{2} \int_0^{t-|x-L|} z(L,s) \, ds \quad \mbox{for} \;
  (x,t) \in Q .
\end{equation*}
In order to determine the yet unknown density function $z$, with $(z_0,z_L):=
(z(0,\cdot),z(L,\cdot))$, we consider the boundary integral equation (BIE) for 
$x \to 0$,
\begin{equation}\label{eq:V0}
(\OV_0z)(t) :=
  \frac{1}{2} \int_0^t z_0(s) \, ds +
\frac{1}{2} \int_0^{t-L} z_L(s) \, ds = g_0(t) \quad \mbox{for} \;
t \in (0,T),
\end{equation}
and for $ x \to L$,
\begin{equation}\label{eq:VL}
  (\OV_Lz)(t) :=
\frac{1}{2} \int_0^{t-L} z_0(s) \, ds +
  \frac{1}{2} \int_0^t z_L(s) \, ds = g_L(t) \quad \mbox{for} \; t \in (0,T).
\end{equation}
We write the BIEs (\ref{eq:V0}) and (\ref{eq:VL}) in
compact form, for $z=(z_0,z_L)$, as
\begin{equation}\label{BIE}
  (\OV z)(t) =
  \begin{pmatrix}
           (\OV_0z)(t) \\
           (\OV_Lz)(t)
  \end{pmatrix}
  = \begin{pmatrix}
        \OV_{00} & \OV_{0L} \\
        \OV_{L0} & \OV_{LL}
    \end{pmatrix}
    \begin{pmatrix}
        z_0 \\
        z_L
    \end{pmatrix}
    (t)=
    \begin{pmatrix}
        g_0(t) \\
        g_L(t)
    \end{pmatrix}
    = g(t), \quad t \in (0,T)
\end{equation}
and consider the spaces $L^2(\Sigma):=L^2(0,T)\times L^2(0,T)$, and
\begin{align}
    H^1_{0,}(\Sigma)
    &:=\{ v=(v_0, v_L)^T : v_0, v_L \in H^1(0,T), v_0(0)=v_L(0)=0\},\\
H^1_{,0}(\Sigma)
    &:=\{ v =(v_0,v_L)^T: v_0, v_L \in H^1(0,T), v_0(T)=v_L(T)=0\},
\end{align}
endowed with the norms $\norm{v}_{H_{0,}^1(0,T)} = \norm{v}_{H_{,0}^1(0,T)} = 
\norm{v'}_{L^2(0,T)}$. For $0\leq s\leq 1$, we define $H^{s}_{0,}(\Sigma)$, and 
$H^{s}_{,0}(\Sigma)$ by interpolation as
\begin{align*}
  H^{s}_{0,}(\Sigma) :=
  [L^2(\Sigma),H^1_{0,}(\Sigma)]_{\theta=s}, \qquad
  H^{s}_{,0}(\Sigma) :=
  [L^2(\Sigma),H^1_{,0}(\Sigma)]_{\theta=s}.
\end{align*}

In $1d$, the weakly singular boundary integral operator $\OV$ is an isomorphism
when mapping $\OV: [H^s_{,0}(\Sigma)]^\prime \to H^{1-s}_{0,}(\Sigma)$, for $0
\leq s\leq 1$ \cite[Sect.~2]{SUZ21}.

From \eqref{BIE}, we see that we can split $\OV$ into the operators 
$\OV_D$ and $V_{OD}$ as
\begin{equation}\label{eq:decouplingV}
  \OV_D := \begin{pmatrix}
    \OV_{00} & 0\\ 0 &\OV_{LL}
  \end{pmatrix},
  \quad \text{and}\quad 
  \OV_{OD} := \begin{pmatrix}
    0 & \OV_{0L}\\ \OV_{L0} & 0
  \end{pmatrix}. 
\end{equation}
Introducing
\begin{equation}
\label{eq:ntimeslices}
 n := \min\{ m\in
\mathbb{N} : T\leq mL\},
\end{equation}
and the time slices
\begin{equation}
\label{eq:timeslices} 
T_j := ((j-1)L,j L) \text{ for }j =1,\ldots,n-1, \quad \text{ and }\:T_n:=((n-1)L,T);
\end{equation}
we note that the off-diagonal term $\OV_{OD}$ contains the cross-boundary terms, 
which are zero if $t<L$, i.e., $\OV_{OD}$ is present only when there 
are multiple time slices and couples them. The diagonal term $\OV_D$ behaves like
applying $\OV$ on a single time slice, independent
of $n$. Moreover, we have that 
\cite[Lemma 2.2]{HLS23}:  
\begin{equation}\label{eq:norm-via-diagonal-operator}
  \frac{1}{2}\norm{z}_{[H^1_{,0}(\Sigma)]'}=\norm{\OV_D z}_{L^{2}(\Sigma)}, 
  \qquad \forall z\in [H^1_{,0}(\Sigma)]'.  
\end{equation} 
Let $\Sigma_j$, $j=1,\ldots,n$, denote the lateral boundary restricted to the
$j^{th}$ time-slice in time $T_j$. Then there holds the shifting relation
\cite[Eq.~(2.9)]{HLS23}
\begin{equation}\label{eq:shifting-time-slices}
\norm{\OV_D z}_{L^{2}(\Sigma_{j-1})} = \norm{\OV_{OD}z}_{L^{2}(\Sigma_j)}, 
\quad j=2,\ldots,n-1
\end{equation}
and 
\begin{equation*}
\norm{\OV_D z}_{L^{2}(\Sigma_{n-1})} \geq \norm{\OV_{OD}z}_{L^{2}(\Sigma_{n})},   
\end{equation*}
with equality if $T=nL$. 
With this, we can easily compute the boundedness  
\begin{equation*}
  \norm{\OV z}_{L^2(\Sigma)} 
  \leq \norm{\OV_D z}_{L^2(\Sigma)} + \norm{\OV_{OD}
  z}_{L^2(\Sigma)} 
  \leq  \frac{1}{2}\norm{z}_{[H^1_{,0}(\Sigma)]'} + \sum_{j=2}^{n-1
  } \norm{\OV_D z}_{L^2(\Sigma_j)} \leq \norm{z}_{[H^1_{,0}(\Sigma)]'}.
\end{equation*}

\subsection{Characterization of the Optimal Test Space}
\label{ssec:optTestV}
As already mentioned, we will apply the abstract framework outlined in
Section \ref{sec:abstract-setting} to investigate stable discrete formulations.
Let us begin by considering $\OV :[H^1_{,0}(\Sigma)]'\to L^2(\Sigma)$ and some
given function $g\in L^2(\Sigma)$. We aim to solve: Find $z\in [H^1_{,0}(\Sigma
)]'$ such that\footnote{Here we could simply use the $L^2$-inner product,
but we prefer to write the duality pairing to make the connection with
Sections~\ref{sec:abstract-setting} and \ref{sec:NoConstraints} easier. However,
for later convenience, when considering a given subset $D\subset\mathbb{R}^d$
with $d=1,2$, we denote the $L^2(D)$-duality pairing as $\dual{\cdot}{\cdot}_{D
}$ instead of simply $\dual{\cdot}{\cdot}$ as proposed in
Section~\ref{sec:abstract-setting}.}
\begin{equation}\label{eq:wave-operator-equation}
\dual{\OV z}{q}_{\Sigma}=\dual{g}{q}_{\Sigma},\quad\forall q\in L^2
(\Sigma). 
\end{equation}

From \cite[Eq.~(2.11)]{HLS23} we know that $\OV:[H^1_{,0}(\Sigma)]'\to 
L^2(\Sigma)$ is an isomorphism, and there holds the inf-sup stability condition
\begin{equation}\label{eq:inf-sup-optimal}
\alpha_{\text{is}}^{\OV}(n) 
:=  \sin\left(\frac{\pi}{2(2n+1)}\right) 
\leq \inf_{z\in [H^1_{,0}(\Sigma)]'}\sup_{0\neq q\in L^2(\Sigma)}
\frac{\dual{\OV z}{q}_{\Sigma}}{\norm{z}_{[H^1_{,0}(\Sigma)]'}\norm{q}_{L^2(
\Sigma)}},
\end{equation}
where $n$ is defined as in \eqref{eq:ntimeslices}.

Now, let us write our mesh as $\Sigma_h = (\Sigma_h^0,\Sigma_h^L)$, where
$\Sigma_h^\ast = \bigcup_{\ell=1}^{N_\ast} \tau_\ell^\ast$, $\ast\in\{0,L\},$
denotes a decomposition of $\Sigma^\ast$ into boundary elements (BE) and consider
the BE space of \emph{piecewise constant functions} $S^0(\Sigma_h) = \text{span}
\{\varphi_\ell^0\}_{\ell=1}^{N:=N_0+N_L}$, and the BE space of \emph{piecewise
linear functions with zero initial condition} $S^1_{0,}(\Sigma_h) = \text{span}
\{\varphi_k^1\}_{k=1}^N$. We introduce the following constraint on the mesh.

\begin{assumption}\label{Assumption:mesh}
Let us assume $T = nL$ for some $n\in \mathbb{N}$,
and that the mesh on the boundary restricted to a time-slice
corresponds to the mesh on the opposite boundary restricted to the
subsequent time-slice, i.e., 
\begin{align}
    \Sigma^0_{h|T_j}= \Sigma^L_{h|T_{j+1}}, \text{ and }\: \Sigma^L_{h|T_j}=
    \Sigma^0_{h|T_{j+1}}, \quad \forall j=1,\dots,n-1.
\end{align}
\end{assumption}
\begin{figure}[H]
  \vspace{-.25cm}
      \centering
    \begin{tikzpicture}

\shade[left color=white,right color=lightgray] (0,0) rectangle (5,5) node[pos=.5] {$1^{\text{st}}$ time-slice};
\shade[left color=lightgray,right color=white] (0,5) rectangle (5,10) node[pos=.5] {$2^{\text{nd}}$ time-slice};
\def\lastx{-0.05}
\draw[black] (0,0)--(0,10);
\draw[black] (5,0)--(5,10);
\foreach \x[remember =\x as \lastx] in {0,0.4, 1, 2, 4 ,5}{
    \draw (0,\x) node[rectangle, fill =orange, scale = 0.4, draw=black, double] {};
    \draw (5,\x+5) node[rectangle, fill =orange, scale = 0.4, draw= black,double] {};
    }
\def\lastx{-0.05}
\foreach \x[remember =\x as \lastx] in {0,1,2.5,2.8,4.3,5}{
    \draw (0,\x+5) node[circle, fill =blue, scale = 0.4, draw= black, double] {};
    \draw (5,\x) node[circle, fill =blue, scale = 0.4, draw= black, double] {};
    }
\draw[dotted] (-1.5,0)  -- (0,0) node[pos=-0.5] {$t=0$};
\draw[dotted] (-1.5,5)  -- (0,5) node[pos=-0.5] {$t=L$};
\draw[dotted] (-1.5,10)  -- (0,10) node[pos=-0.5] {$t=2L$};
\draw[dotted] (0,-1)  -- (0,0) node[pos=-0.3] {$x=0$};
\draw[dotted] (5,-1)  -- (5,0) node[pos=-0.3] {$x=L$};

\end{tikzpicture}%
    \caption{Example of mesh fulfilling Assumption~\ref{Assumption:mesh}.}
    \label{fig:Example_mesh}
  \end{figure}
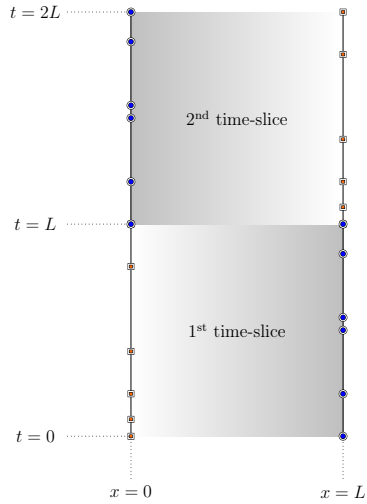
  
Let $S^0(\Sigma_h)= \text{span}\{\varphi^0_\ell\}_{\ell=1}^N$. Following 
Section~\ref{sec:abstract-setting}, we obtain:
\begin{theorem}\label{thm:YopV}
Under Assumption \ref{Assumption:mesh} on the mesh, for any $g = \OV z\in L^2(
\Sigma)$ the discrete variational formulation:
find $z_h\in S^0(\Sigma_h)$ such that
\begin{equation}\label{eq:DVF-H-1,0 to L2}
\dual{\OV z_h}{q_h}_{\Sigma} = \dual{g}{q_h}_{\Sigma},\quad \forall
q_h\in S_{0,}^{1}(\Sigma_h)
\end{equation}
admits a unique solution satisfying the estimate 
\begin{equation*}
  \norm{z-z_h}_{[H^1_{,0}(\Sigma)]'} \leq \frac{1}{\alpha_{\text{is}}^{\OV}(n)}
  \inf_{v_h\in S^0(\Sigma_h)}\norm{z-v_h}_{[H^1_{,0}(\Sigma)]'}.
\end{equation*}
\end{theorem}
\proof{
Let us consider the spaces $X = [H^1_{,0}(\Sigma)]'$ and $Y = L^2(\Sigma)$, and the
bilinear form $\bb(\cdot,\cdot): [H^1_{,0}(\Sigma)]'\times L^2(\Sigma) \to
\mathbb{R}$ defined as
\begin{equation*}
  \bb(z,q) := \dual{\OV z}{q}_{\Sigma}.
\end{equation*}
Then, the discrete variational formulation \eqref{eq:DVF-H-1,0 to L2} can be
rewritten as: Find $z_h\in X_h= S^0(\Sigma_h)$ such that
\begin{equation*}
  \bb(z_h,q_h) = \dual{g}{q_h}_{\Sigma},\quad \forall q_h\in Y_h= S_{
  0,}^1(\Sigma_h).
\end{equation*}  
If we can show that $Y_h= Y_h^{\text{opt}}$, the result follows from
Theorem~\ref{thm:optimal setting main results}. In our case the Riesz isomorphism
$\OR_Y : L^2(\Sigma) \to L^2(\Sigma)$ can be chosen to be the identity. Recall
that the optimal test space is then given as $Y_h^{\text{opt}} = \OR_Y^{-1}
\OV (X_h) = \OV (S^0(\Sigma_h))$.
So, we compute the action of the weakly singular operator on piecewise 
constant functions. For clarity, let us write $S^0(\Sigma_h) = S^0(\Sigma_{h}^0
)\times S^0(\Sigma_{h}^L)$, where $\Sigma_{h}^0$ and $\Sigma_{h}^L$ are the 
decompositions of the boundary at $x=0$ and $x=L$, respectively. Moreover, for
each $\varphi_\ell^0\in S^0(\Sigma_h)$, let $\text{supp}(\varphi_\ell^0):=
(t_{min},t_{max})$. Then, we either have $(t_{min},t_{max})\subset\overline
\Sigma_{h}^0$ or $(t_{min},t_{max})\subset \overline \Sigma_{h}^L$, and we
distinguish the following two cases:\\
\noindent\underline{\textbf{Case 1:} $T \leq L$.} 
Recall the splitting $\OV = (\OV_0,\OV_L)$, as in \eqref{eq:V0} 
and \eqref{eq:VL}.
If $(t_{min},t_{max})\subset \overline{\Sigma}_{h}^0$, then $(\OV_L \varphi_\ell^0
)(t)=0$, and
\begin{equation}\label{eq:Yopt-1-time-slice-0}
(\OV_0 \varphi_\ell^0)(t) = \frac{1}{2}\int_0^t \varphi_\ell^0(0,s)\,ds = 
\begin{Bmatrix*}[l]
    0,& t\leq t_{min},\\
    \frac{1}{2}(t-t_{min}), & t_{min}<t<t_{max},\\[0.03in]
\frac{1}{2}(t_{max}-t_{min}), & t_{max}\leq t,\\
\end{Bmatrix*}\in S^1_{0,}(\Sigma_{h}^0).
  \end{equation}
Analogously, if $(t_{min},t_{max})\subset \overline{\Sigma}_{h}^L$, then $(\OV_0
\varphi_\ell^0)(t)=0$, and
\begin{equation}\label{eq:Yopt-1-time-slice-L}
(\OV_L \varphi_\ell^0)(t) = \frac{1}{2}\int_0^t \varphi_\ell^0(L,s)\,ds = 
\begin{Bmatrix*}[l]
    0,& t\leq t_{min},\\
\frac{1}{2}(t-t_{min}), & t_{min}<t<t_{max},\\[0.03in]
\frac{1}{2}(t_{max}-t_{min}), & t_{max}\leq t,\\
\end{Bmatrix*}\in S^1_{0,}(\Sigma_{h}^L).
\end{equation}
Hence, we have that $Y_h^{\text{opt}} = V(S^0(\Sigma_h))\subset S^1_{0,}(\Sigma_h
)$ and since $\text{dim}(V(S^0(\Sigma_h))) = \text{dim}(S^0(\Sigma_h)) = \text{
dim}(S_{0,}^1(\Sigma_h))$, the equality $Y_h^{\text{opt}} = S_{0,}^1(\Sigma_h)$
holds.
  
\noindent
\underline{\textbf{Case 2}: $T>L$, i.e., we have at least two time slices.} 
Here, in addition to the expressions in \eqref{eq:Yopt-1-time-slice-0} and 
\eqref{eq:Yopt-1-time-slice-L} we have to account for the cross-boundary terms
induced by $\OV_{0L}$
and $\OV_{L0}$. We compute  
\begin{itemize}
 \item for $(t_{\text{min}},t_{\text{max}})\subset \overline{\Sigma}_{h}^0$, that
\begin{equation}\label{eq:Yopt-more-time-slice-0}
\frac{1}{2}\int_0^{t-L} \varphi_\ell^0(L,s)\,ds 
= \begin{cases}
    0,& t\leq t_{min},\\
\frac{1}{2}(t-L-t_{min}), & t_{min}+L<t<t_{max}+L,\\
\frac{1}{2}(t_{max}-t_{min}), & t_{max}+L\leq t.
\end{cases}
\end{equation}

\item and for $(t_{\text{min}},t_{\text{max}})\subset \overline{\Sigma}_{h}^L$,
we add
\begin{equation}\label{eq:Yopt-more-time-slice-L}
\frac{1}{2}\int_0^{t-L} \varphi_\ell^0(0,s)\,ds 
= \begin{cases}
    0,& t\leq t_{min},\\
\frac{1}{2}(t-L-t_{min}), & t_{min}+L<t<t_{max}+L,\\
\frac{1}{2}(t_{max}-t_{min}), & t_{max}+L\leq t.
\end{cases}
  \end{equation}
\end{itemize}
Note that \eqref{eq:Yopt-more-time-slice-0} and \eqref{eq:Yopt-more-time-slice-L} 
induce a relation between the meshes $\Sigma_{h}^0$ and $\Sigma_{h}^L$. Thus, 
we can only guarantee that $\OV_0 \varphi_\ell^0\in S^1_{0,}(\Sigma_{h}^0)$ and 
$\OV_L \varphi_\ell^0\in S^1_{0,}(\Sigma_{h}^L)$ if $\Sigma_{h}^0 
= \Sigma_{h}^L-L$, which is exactly Assumption \ref{Assumption:mesh}. 
Then, as in Case 1, we get $Y_h^{\text{opt}}= S_{0,}^1(\Sigma_h)$, which
finishes the proof.
}
Thus, under Assumption \ref{Assumption:mesh}, the optimal test space can be
characterized and realized explicitly. This assumption admits a clear physical
interpretation. Waves emanating from the left boundary reach the right boundary
after the first time slice. To ensure optimal stability, these waves must be
resolved accurately, which imposes the restriction on the mesh.

Nevertheless, stability can still be expected if the arriving waves are
approximated ``sufficiently well''. In other words, even without the mesh
constraint, the method may remain stable. This is confirmed in the following
subsection, where we also discuss alternative pairings that appear natural
but lead to unstable formulations.

\subsection{Dating guide for discrete spaces: Stable and unstable pairs}
\label{sec:mainResults}
We now use the developed framework to show (in-)stability of canonical pairs
of trial and test spaces.

\subsubsection{Why you shouldn't date yourself: The unstable pair
$S^0(\Sigma_h) - S^0(\Sigma_h)$}
At first glance, it seems natural to choose $X_h = S^0(\Sigma_h)\subset [H^1_{
,0}(\Sigma)]'$ and $Y_h = S^0(\Sigma_h)\subset L^2(\Sigma)$ as conforming trial
and test spaces. However, this combination has already been reported to be
unstable in the literature \cite[Chap.~2.1]{Gua10}. Here, we will use the
framework from Section~\ref{sec:abstract-setting} to finally prove it.

Let us consider the variational formulation: Find $z_h\in
S^0(\Sigma_h)$ such that 
\begin{equation}
\bb_0(z_h,q_h) := \dual{\OV z_h}{q_h}_{\Sigma} = \dual{g}{q_h}_{
\Sigma},\quad \quad\forall  q_h\in S^0(\Sigma_h).
\end{equation}   
For the sake of simplicity, we assume that 
$T\leq L$. The rationale is that if stability fails to hold for one time slice, 
it will also fail when considering more than one. Furthermore, it is sufficient 
to consider one of the boundary parts, i.e., either $x=0$ or $x=L$. With that 
in mind, let us decompose $(0,T)$ into $N$ elements with uniform mesh size $h 
= \frac{T}{N}$, i.e., each element is defined as $\tau_\ell :=(t_{\ell-1},
t_\ell)$ with $t_{\ell}:= h\ell$ for $\ell=1,\ldots,N$. We denote the 
corresponding space of \emph{piecewise constant functions} by $S^0_h(0,T) = 
\text{span}\{\varphi_\ell^0\}_{\ell=1}^N$ and the 
related space of \emph{piecewise linear functions with zero initial 
condition} by $S_{h;0,}^1(0,T) = \text{span}\{\varphi_k^1\}_{k=1}^N$.
Now, consider $Z_h(t) = \int_0^t z_h(s)\, ds \in S_{h;0,}^1(0,T)$. Then we have
\begin{equation*}
(\OV z_h)(t) = \frac{1}{2}\int_0^t z_h(s)\,ds = \frac{1}{2}\int_0^t Z_h'(s)\,
ds = \frac{Z_h(t)}{2},
\end{equation*}  
and 
\begin{equation*}
\bb_0(z_h,q_h) = \int_0^T (\OV z_h)(t)q_h(t)\, dt = \frac{1}{2}\int_0^T Z_h(t)
q_h(t)\, dt.
\end{equation*}
Moreover, using that $Z_h(0)=0$ and $\partial_t:H^1_{0,}(0,T)\to L^2(0,T)$ is
an isomorphism, it holds that
\begin{equation}\label{eq:H-1 to L2 norm}
\begin{aligned}
\norm{z_h}_{[H^1_{,0}(0,T)]'} &= \sup_{0\neq w\in H^1_{,0}(0,T)}\frac{\int_0^T
z_h(t)w(t)\, dt }{\norm{w'}_{L^2(0,T)}} = \sup_{0\neq w\in H^1_{,0}(0,T)}\frac{
\int_0^T Z_h'(t)w(t)\, dt}{\norm{w'}_{L^2(0,T)}}\\
& = \sup_{0\neq w\in H^1_{,0}(0,T)}\frac{-\int_0^T Z_h(t)w'(t)\, dt}{\norm{w'}_{
L^2(0,T)}}\\
&= \sup_{0\neq -w'=q\in L^2(0,T)} \frac{\int_0^T Z_h(t)q(t)\, dt }{
\norm{q}_{L^2(0,T)}} = \norm{Z_h}_{L^2(0,T)}.
\end{aligned}
\end{equation}
So, the discrete inf-sup constant is characterized as 
\begin{equation}\label{eq:discrete inf sup unstable}
c_0^{\OV} = \inf_{Z_h\in S_{h;0,}^1(0,T)}\sup_{0\neq q_h\in S_h^0(0,T)}\frac{
\frac{1}{2}\int_0^T Z_h(t)q_h(t)\, dt}{\norm{Z_h}_{L^2(0,T)}\norm{q_h}_{L^2(0,
T)}}.
\end{equation}
To get an intuitive feeling for why the pairing is unstable, we note that
\begin{equation*}
  \int_0^T Z_h(t)q_h(t)\, dt =  \int_0^T (\O\OQ_h^0 Z_h)(t)q_h(t)\,  dt,
\end{equation*}
where $\O\OQ_h^0:L^2(0,T)\to S_h^0(0,T)$ denotes the $L^2$-projection onto
piecewise constant functions. Choosing $\bar q_h = \OQ_h^0Z_h \in S_h^0(0,T)$,
we obtain
\begin{align*}
\frac{1}{2}\norm{\OQ_h^0Z_h}_{L^2(0,T)}\norm{\bar q_h}_{L^2(0,T)} &= \frac{1}{2}
\norm{\OQ_h^0Z_h}_{L^2(0,T)}^2 = \frac{1}{2}\int_0^T (\OQ_h^0 Z_h)(t)\bar q_h(t)
\, dt = \bb_0(z_h,\bar q_h) ,
\end{align*}
and thus
\begin{align*}
\frac{1}{2}\norm{\OQ_h^0 Z_h}_{L^2(0,T)} &\leq \sup_{0\neq q_h\in S_h^0(0,T)}
\frac{\bb_0(z_h,q_h)}{\norm{q_h}_{L^2(0,T)}} = \sup_{0\neq q_h\in S_h^0(0,T)}
\frac{\frac{1}{2}\int_0^T (\OQ_h^0Z_h)(t)q_h(t)\, dt }{\norm{q_h}_{L^2(0,T)}}\\
& \leq \frac{1}{2}\norm{\OQ_h^0Z_h}_{L^2(0,T)},
\end{align*}
which implies
\begin{equation*}
\frac{1}{2}\norm{\OQ_h^0 Z_h}_{L^2(0,T)} = \sup_{0\neq q_h\in S_h^0(0,T)}
\frac{\bb_0(z_h,q_h)}{\norm{q_h}_{L^2(0,T)}}.
\end{equation*}
If we had an estimate of the form $ \norm{\OQ_h^0 Z_h}_{L^2(0,T)} \geq c\norm{
Z_h}_{L^2(0,T)}$ for all $Z_h\in S_{h;0,}^1(0,T)$ with a constant $c>0$ that is
independent of the mesh size, together with \eqref{eq:H-1 to L2 norm} this would
give discrete inf-sup stability. However, it turns out that such an estimate
cannot hold. This becomes evident when considering the highly oscillatory
functions $$z_h(t) = \varphi_1^0(t)+2\sum_{\ell=2}^N (-1)^{\ell-1}\varphi_\ell^0
(t),$$ for which $\OQ_h^0Z_h$ almost vanishes, while $Z_h$ does not. We depict
this in Figure~\ref{fig:example_unstable}.

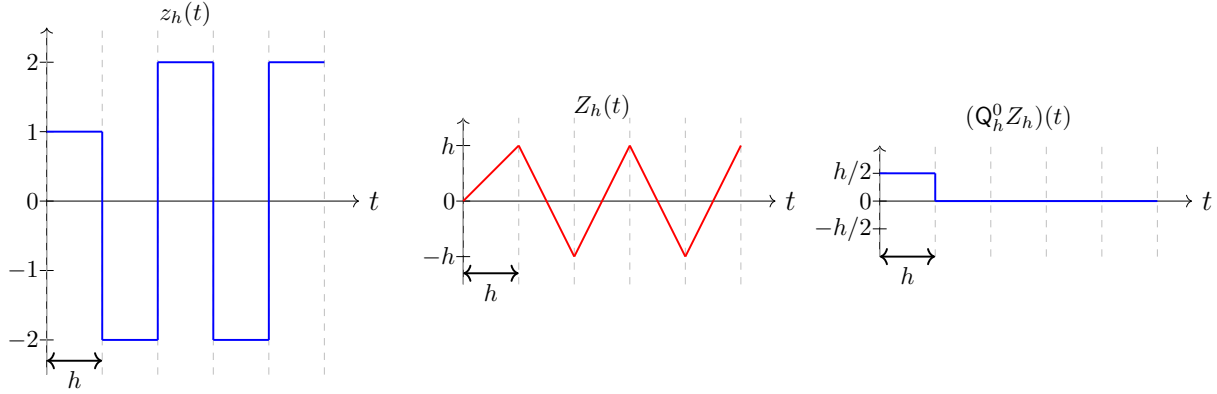
\begin{figure}[H]
  \centering
  \begin{tikzpicture}[scale=1]
    \def\h{0.8}
    \def\n{5}               
    \def\spacebetween{6}

    \begin{scope}
        \draw[->] (0, -2.5) -- (0, 2.5) node[above] {};
        \draw[->] (0, 0) -- ({\n*\h + 0.5}, 0) node[right] {$t$};
        \node at ({\n*\h/2}, 2.7) {\footnotesize $z_h(t)$};

        \foreach \y in {-2,-1,0,1,2} {
            \draw (-0.1,\y) -- (0.1,\y) node[left=2pt] {\footnotesize $\y$};
        }

        \foreach \i in {0,...,\n} {
            \draw[dashed, gray!50] ({\i*\h}, -2.5) -- ({\i*\h}, 2.5);
        }

        \def\valsConst{{1, -2, 2, -2, 2}}

        \foreach \i [evaluate=\i as \xleft using {\i*\h}, evaluate=\i as \xright using {(\i+1)*\h}] in {0,...,4} {
            \pgfmathparse{\valsConst[\i]}
            \let\yval\pgfmathresult

            \draw[thick, blue] (\xleft, \yval) -- (\xright, \yval);

            \ifnum\i<4
                \pgfmathparse{\valsConst[\i+1]}
                \let\ynext\pgfmathresult
                \draw[thick, blue] (\xright, \yval) -- (\xright, \ynext);
            \fi
        }

        \draw[<->, thick] (0, -2.3) -- (\h, -2.3) node[midway, below] {\footnotesize $h$};
    \end{scope}

    \begin{scope}[xshift=\spacebetween cm]
        \draw[->] (0, -1.5*\h) -- (0, 1.5*\h) node[above] {};
        \draw[->] (0, 0) -- ({\n*\h + 0.5}, 0) node[right] {$t$};
        \node at ({\n*\h/2}, 1.7*\h) {\footnotesize $Z_h(t)$};

        \foreach \k/\lab in {-1/-h,0/0,1/h} {
            \draw (-0.1,{\k*\h}) -- (0.1,{\k*\h}) node[left=2pt] {\footnotesize $\lab$};
        }

        \foreach \i in {0,...,\n} {
            \draw[dashed, gray!50] ({\i*\h}, -1.5*\h) -- ({\i*\h}, 1.5*\h);
        }

        \coordinate (prev) at (0,0);
        \foreach \i in {1,...,\n} {
            \pgfmathsetmacro{\x}{\i * \h}
            \pgfmathtruncatemacro{\sign}{mod(\i,2)==1 ? 1 : -1}
            \pgfmathsetmacro{\y}{\sign * \h}
            \draw[thick, red] (prev) -- (\x,\y);
            \coordinate (prev) at (\x,\y);
        }

        \draw[<->, thick] (0, -1.3*\h) -- (\h, -1.3*\h) node[midway, below] {\footnotesize $h$};
    \end{scope}

    \begin{scope}[xshift=2*\spacebetween cm ]
        \draw[->] (0, -\h) -- (0,\h) node[above] {};
        \draw[->] (0, 0) -- ({\n*\h + 0.5}, 0) node[right] {$t$};
        \node at ({\n*\h/2}, 1.5*\h) {\footnotesize $(\mathsf{Q}_h^0Z_h)(t)$};

        \foreach \k/\lab in {-0.5/{-h/2},0/0,0.5/{h/2}} {
            \draw (-0.1,{\k*\h}) -- (0.1,{\k*\h}) node[left=2pt] {\footnotesize $\lab$};
        }

        \foreach \i in {0,...,\n} {
            \draw[dashed, gray!50] ({\i*\h}, -\h) -- ({\i*\h}, \h);
        }

        \def\valsConst{{\h/2, 0, 0, 0, 0}}

        \foreach \i [evaluate=\i as \xleft using {\i*\h}, evaluate=\i as \xright using {(\i+1)*\h}] in {0,...,4} {
            \pgfmathparse{\valsConst[\i]}
            \let\yval\pgfmathresult

            \draw[thick, blue] (\xleft, \yval) -- (\xright, \yval);

            \ifnum\i<4
                \pgfmathparse{\valsConst[\i+1]}
                \let\ynext\pgfmathresult
                \draw[thick, blue] (\xright, \yval) -- (\xright, \ynext);
            \fi
        }

        \draw[<->, thick] (0, -\h) -- (\h, -\h) node[midway, below] {\footnotesize $h$};
    \end{scope}

\end{tikzpicture}
  \caption{Illustration of $z_h$, $Z_h$ and $\OQ_h^0Z_h$ where ``things go
  wrong''. }
  \label{fig:example_unstable}
\end{figure}

In the following, we will compute the exact lower bound, showing how the inf-sup
constant for this case depends on the mesh size.
\begin{theorem}\label{thm:discrete_inf-sup_instable_constants}
It holds that
\begin{equation}\label{eq:discrete inf-sup S0-S0}
c_0^{\OV}(h,T) := \sqrt{\frac{3}{8}\left(\frac{1-\cos\left(\tfrac{h\pi}{2T}\right)
}{2-\cos\left(\tfrac{h\pi}{2T}\right)}\right)} = \inf_{z_h\in S^0(\Sigma_h)}
\sup_{0\neq q_h\in S^0(\Sigma_h)}\frac{\bb_0(z_h,q_h)}{\norm{z_h}_{[H^1_{,0}
(\Sigma)]'}\norm{q_h}_{L^{2}(\Sigma)}}.
\end{equation}
In particular, we have that
\begin{equation*}
  c_0^{\OV}(h,T) = \sqrt{\frac{3}{64}}\pi \frac{h}{T} +o((h/T)^{3/2}) ,\quad
  h\to 0,
\end{equation*}
and the discrete inf-sup stability \eqref{eq:discrete-inf-sup-B} fails to hold
for a mesh independent constant.
\end{theorem}
\proof{
The discrete inf-sup constant characterized by \eqref{eq:discrete inf sup unstable}
can be computed as $c_0^{\OV} = \sqrt{\lambda_{\text{min}}}$, where $\lambda_{
\text{min}}$ denotes the minimal eigenvalue of
\begin{equation}\label{eq:generalized EVP inf-sup}
  \OOB_h^\top (\OOM_h^{0,0})^{-1}\OOB_h \underline{v} = \lambda \OOM_h^{1,1}
  \underline{v},
\end{equation}
with matrices 
\begin{equation*}
\OOB_h[i,j] = \frac{1}{2}\int_0^T \varphi_j^1(t)\varphi_i^0(t)\, dt, \quad
\OOM_h^{0,0}[i,j] = h\delta_{ij},\quad \OOM_h^{1,1}[i,j] = \int_0^T
\varphi_j^1(t)\varphi_i^1(t)\, dt,
\end{equation*}
for $i,j=1,\ldots,N$, see \cite[Remark 3.159]{John:2016}. Here $\delta_{ij}$
denotes the Kronecker delta. For the case at hand, we have that
\begin{equation*}
  \OOB_h=\frac{h}{4}
  \begin{pmatrix}
    1 & 0 & 0 & \cdots & 0\\
    1 & 1 & 0 & \cdots & 0\\
    0 & 1 & 1 & \ddots & \vdots\\
    \vdots & \ddots & \ddots & \ddots & 0\\
    0 & \cdots & 0 & 1 & 1
  \end{pmatrix}
  ,
  \quad 
  \OOM_h^{1,1}=\frac{h}{6}
  \begin{pmatrix}
    4 & 1 & 0 & \cdots & 0\\
    1 & 4 & 1 & \ddots & \vdots\\
    0 & 1 & 4 & \ddots & 0\\
    \vdots & \ddots & \ddots & \ddots & 1\\
    0 & \cdots & 0 & 1 & 2
  \end{pmatrix}.
\end{equation*}
Now, consider 
\begin{equation*}
  \theta_k = \frac{1}{N}\left(k\pi -\frac{\pi}{2}\right),\quad k=1,\ldots,N,
\end{equation*} 
and the vector $\underline{v}_k\in\mathbb{R}^N$ with entries
\begin{equation*}
  \underline{v}_k[j] = \sin(j\theta_k),\quad j=1,\ldots,N.  
\end{equation*}
Let us slightly abuse notation and write $\underline{v}_k[0] = \sin(0)=0$.
Then, employing the trigonometric relation
\begin{equation*}
  \sin((j-1)\theta_k) + \sin((j+1)\theta_k) = 2\cos(\theta_k)\sin(j\theta_k), 
\end{equation*}
we compute for $j=1,\ldots,N-1$ that
\begin{equation*}
  (\OOM_h^{1,1}\underline{v}_k)[j] = \frac{h}{6}\left(\underline{v}_k[j-1] +
  4\underline{v}_k[j] + \underline{v}_k[j+1]\right) = \frac{h}{6}\left(2\cos(
  \theta_k) + 4\right)\underline{v}_k[j].
\end{equation*}
Using that $N\theta_k = k\pi-\tfrac{\pi}{2}$, we get for $j=N$ that
\begin{equation*}
  \underline{v}_k[N-1] = \sin(N\theta_k-\theta_k) = \sin(N\theta_k)\cos(
  \theta_k) - \cos(N\theta_k)\sin(\theta_k) = \sin(N\theta_k)\cos(\theta_k),
\end{equation*} 
and thus 
\begin{equation*}
  (\OOM_h^{1,1}\underline{v}_k)[N] = \frac{h}{6}\left(\underline{v}_k[N-1]
  + 2\underline{v}_k[N]\right) = \frac{h}{6}\left(\cos(\theta_k) + 2\right)
  \underline{v}_k[N].
\end{equation*}
So, altogether, we obtain
\begin{equation*}
\OOM_h^{1,1}\underline{v}_k = \frac{h}{6}(2+\cos(\theta_k))\begin{pmatrix}
  2\underline{v}_k[1] \\
  \vdots\\
  2\underline{v}_k[N-1]\\
  \underline{v}_k[N]
  \end{pmatrix}.
\end{equation*}
Noting that 
\begin{equation*}
  \OOB_h^\top \OOB_h = \frac{h^2}{16}
\begin{pmatrix}
2 & 1 & 0 & \cdots & 0\\
1 & 2 & 1 & \ddots & \vdots\\
0 & 1 & 2 & \ddots & 0\\
\vdots & \ddots & \ddots & \ddots & 1\\
0 & \cdots & 0 & 1 & 1
\end{pmatrix},  
\end{equation*}
we can analogously compute 
\begin{equation*}
  \OOB_h^\top (\OOM_h^{0,0})^{-1}\OOB_h \underline{v}_k = \frac{h}{16}(1+\cos(
  \theta_k)) \begin{pmatrix}
    2\underline{v}_k[1]\\
    \vdots\\
    2\underline{v}_k[N-1]\\
    \underline{v}_k[N]
  \end{pmatrix},
\end{equation*}
and thus observe that $\underline{v}_k$ is an eigenvector of
\eqref{eq:generalized EVP inf-sup} with eigenvalue
\begin{equation*}
  \lambda_k= \frac{3}{8}\left(\frac{1+\cos(\theta_k)}{2+\cos(\theta_k)}\right).
\end{equation*}
Furthermore, the minimal eigenvalue is attained at $k=N$. It holds that
$\theta_N= \pi-\frac{\pi}{2N}$ and therefore
\begin{equation*}
  \cos(\theta_N) = -\cos\left(\frac{\pi}{2N}\right),
\end{equation*} 
from which we conclude that
\begin{equation*}
\lambda_{\text{min}} = \frac{3}{8}\left(\frac{1-\cos\left(\frac{\pi}{2N}\right)
}{2-\cos\left(\frac{\pi}{2N}\right)}\right).
\end{equation*}
Last but not least, we use the expansion $\cos(x) = 1-\frac{x^2}{2} + o(x^3)$
to obtain the asymptotic behavior, recall $h=T/N$
\begin{equation*}
\lambda_{\text{min}} = \frac{3}{8}\left(\frac{\frac{\pi^2}{8N^2}}{1+\frac{\pi^2
}{8N^2}}\right) + o(N^{-3}) = \frac{3\pi^2}{8(8N^2 + \pi^2)} + o(N^{-3}) =
\frac{3\pi^2}{64}\frac{h^2}{T^2} + o((h/T)^{3}).
\end{equation*}
}

\subsubsection{Why someone more refined could be the answer: The stable pair
$S^0(\Sigma_h) - S^0(\Sigma_{h/m})$, $m\in\mathbb{N}_{\geq 2}$}
\label{sssec:stable_LS}
Evidently, the issue of the last section can be resolved when the test space
$Y_h$ is able to resolve linear functions ``accurately enough''. In \cite{HLS23},
we already proved that piecewise constants on a finer mesh are sufficient to
do so. More precisely, let $\Sigma_h$ be a given mesh and set $X_h =S^0(
\Sigma_h)$. Furthermore, define $Y_h = S^0(\Sigma_{h/m})$, where $m\in
\mathbb{N}$ and $\Sigma_{h/m}$ denotes the mesh obtained by uniformly refining
each element in $\Sigma_{h}$ into $m$ elements of mesh size $h/m$. We consider
the variational formulation: Find $z_h\in S^0(\Sigma_h)$ such that
\begin{equation}\label{eq:DVF-with-m}
\bb_{0,m}(z_h,q_{h/m}) := \dual{\OV z_h}{q_{h/m}}_{\Sigma} = \dual{g}{q_{h/m}
}_{\Sigma},\quad \forall q_{h/m}\in S^0(\Sigma_{h/m}).
\end{equation}
Then, by \cite[Theorem 3.2]{HLS23}, \eqref{eq:DVF-with-m} admits a unique solution,
if $m\geq 3$. The following result shows that actually $m\geq 2$ is sufficient
and gives a sharp discrete inf-sup constant.
\begin{theorem}\label{thm:discrete inf sup m}
Let $m\in\mathbb{N}$, $m\geq 2$ and let $n \in\mathbb{N}$ denote the number of
time slices $T_j$ as defined in \eqref{eq:timeslices}. Moreover, let Assumption
\ref{Assumption:mesh} on the mesh hold true. Then
\begin{equation*}
c_{0}^{\OV}(m,n):= \alpha_{\text{is}}^{\OV}(n) \left(1-\frac{1}{m^2}\right)^{1/2}
\leq\inf_{z_h\in S^0(\Sigma_h)}\sup_{0\neq q_{h/m}\in S^0(\Sigma_{h/m})}\frac{
\bb_{0,m}(z_h,q_{h/m})}{\norm{z_h}_{[H^1_{,0}(\Sigma)]'}\norm{q_{h/m}}_{L^{2}(
\Sigma)}}.
\end{equation*} 
\end{theorem}
\proof{
First, let us proceed as in the proof of \cite[Theorem 2.3]{HLS23}, in order to
reduce the discussion to a single time slice. Therefore, first note that, by
self-duality of $L^2(\Sigma)$, we have that
\begin{equation*}
\sup_{0\neq q_{h/m}\in S^0(\Sigma_{h/m})} \frac{\bb_{0,m}(z_h,q_{h/m})}{\norm
{q_{h/m}}_{L^2(\Sigma)}} 
= \sup_{0\neq q_{h/m}\in S^0(\Sigma_{h/m})} \frac{
\dual{\OV z_h}{q_{h/m}}_{\Sigma}}{\norm{q_{h/m}}_{L^2(\Sigma)}} 
= \norm{
\OQ_{h/m}^0 \OV z_h}_{L^2(\Sigma)},
\end{equation*}
where $\OQ_{h/m}^0:L^2(\Sigma)\to S^0(\Sigma_{h/m})$ denotes the $L^2$-projection
onto piecewise constants on $\Sigma_{h/m}$. Now, splitting the weakly singular
operator $\OV$ into diagonal and off-diagonal parts as in \eqref{eq:decouplingV},
gives
\begin{align*}
\norm{\OQ_{h/m}^0\OV z_h}_{L^2(\Sigma)}^2 
&= \norm{\OQ_{h/m}^0(\OV_D + \OV_{OD}
)z_h}_{L^2(\Sigma)}^2\\
& = \norm{\OQ_{h/m}^0\OV_D z_h}_{L^2(\Sigma)}^2 + \norm{\OQ_{h/m}^0 \OV_{OD}z_h
}_{L^2(\Sigma)}^2 \\
  & + 2\dual{\OQ_{h/m}^0 \OV_D z_h}{\OQ_{h/m}^0 \OV_{OD}z_h}_{\Sigma}.
\end{align*}
Due to Assumption \ref{Assumption:mesh} on the mesh, the relation
\eqref{eq:shifting-time-slices} carries over to the discrete setting, i.e.,
\begin{equation*}
\norm{\OQ_{h/m}^0 \OV_{OD}z_h}_{L^2(\Sigma_j)} = \norm{\OQ_{h/m}^0\OV_D z_h}_{
L^2(\Sigma_{j-1})},\quad \forall j = 2,\ldots,n,
\end{equation*} 
and thus 
\begin{align*}
&\norm{\OQ_{h/m}^0\OV z_h}_{L^2(\Sigma)}^2 \\
&= \sum_{j=1}^n\norm{\OQ_{h/m}^0\OV_D z_h}_{L^2(\Sigma_j)}^2 + \sum_{j=2}^n\norm{
\OQ_{h/m}^0 \OV_{OD}z_h}_{L^2(\Sigma_j)}^2 
+ 2\sum_{j=2}^n\dual{\OQ_{h/m}^0 \OV_D
z_h}{\OQ_{h/m}^0 \OV_{OD}z_h}_{\Sigma_j}\\
&\geq \sum_{j=1}^n \norm{\OQ_{h/m}^0\OV_D z_h}_{L^2(\Sigma_j)}^2 
+ \sum_{j=2}^n
\norm{\OQ_{h/m}^0 \OV_D z_h}_{L^2(\Sigma_{j-1})}^2 \\
&\hphantom{\geq \sum_{j=1}^n \norm{\OQ_{h/m}^0\OV_D z_h}_{L^2(\Sigma_j)}^2} - 2\sum_{j=2}^n \norm{\OQ_{
h/m}^0 \OV_D z_h}_{L^2(\Sigma_j)}\norm{\OQ_{h/m}^0 \OV_D z_h}_{L^2(\Sigma_{j-1})}.
\end{align*}
Introducing the vector $\underline{v}[j] = \norm{\OQ_{h/m}^0 \OV_D z_h}_{L^2(
\Sigma_j)}$, $j=1,\ldots,n$, this is equivalent to
\begin{equation}
\norm{\OQ_{h/m}^0 \OV z_h}_{L^2(\Sigma)}^2 \geq \underline{v}^\top \cdot \OOA \underline{v}
\end{equation}
where 
\begin{equation*}
  \OOA =
  \begin{pmatrix}
    2 & -1 & 0 & \cdots & 0\\
    -1 & 2 & -1 & \ddots & \vdots\\
    0 & -1 & 2 & \ddots & 0\\
    \vdots & \ddots & \ddots & \ddots & -1\\
    0 & \cdots & 0 & -1 & 1
  \end{pmatrix},
\end{equation*} 
which is exactly the stiffness matrix for piecewise linear, globally
continuous functions with homogeneous Dirichlet datum on the left boundary, when
using Lagrange finite elements and $h=1$.
Its minimal eigenvalue is known to be $\lambda_{\text{min}}(\OOA) = 4\sin\left(
\frac{\pi}{2(2n+1)}\right)^2 = 4\alpha_{\text{is}}^{\OV}(n)^2$ and thus
\begin{align*}
\norm{\OQ_{h/m}^0 \OV z_h}_{L^2(\Sigma)}^2 \geq \lambda_{\text{min}}(\OOA)
\underline{v}^\top\cdot \underline{v} &= 4\alpha_{\text{is}}^{\OV}(n)^2 \sum_{i=1
}^n \norm{\OQ_{h/m}^0 \OV_D z_h}_{L^2(\Sigma_i)}^2 \\
&= 4\alpha_{\text{is}}^{\OV}(n)^2 \norm{\OQ_{h/m}^0 \OV_D z_h}_{L^2(\Sigma)}^2.
\end{align*}
Now, since the diagonal part $\OV_D$ of the operator acts either on the left
boundary $\Sigma^0:=\lbrace x=0\rbrace\times(0,T)$ or on the right boundary
$\Sigma^L:=\lbrace x=L\rbrace\times(0,T)$, no coupling occurs. Therefore, it is
sufficient to consider the stability on one time slice, i.e., to assume $T\leq L$.
This, in turn, is guaranteed if stability holds on an interval $(0,T)$,
representing one of the boundary parts $\Sigma^0$ or $\Sigma^L$ (as one would
proceed analogously for the other part).

Recalling from \eqref{eq:norm-via-diagonal-operator} that $\norm{\OV_D z_h}_{L^2
(0,T)}^2 = \frac{1}{4}\norm{z_h}_{[H^1_{,0}(0,T)]'}^2$, it suffices to show an
estimate of the form
\begin{equation*}
  \norm{\OQ_{h/m}^0 \OV_D z_h}_{L^2(0,T)}^2 \geq c(m)^2\norm{\OV_D z_h}_{L^2(
  0,T)}^2. 
\end{equation*}
We note that we can show this if for each element $\tau_\ell = (t_{\ell-1},t_\ell
) \in \Sigma_h$ the local relation
\begin{equation}\label{eq:auxiliary-local-relation}
  \norm{\OQ_{h/m}^0 \OV_D z_h}_{L^2(\tau_\ell)}^2 \geq c_{\tau_\ell}(m)^2 \norm{
  \OV_D z_h}_{L^2(\tau_\ell)}^2,\quad \ell=1,\ldots,N,
\end{equation}
holds. This is because in that case, for $c(m): = \min_{\ell=1,\ldots,N}c_{
\tau_\ell}(m)$, we get
\begin{align*}
\norm{\OQ_{h/m}^0 \OV_D z_h}_{L^2(0,T)}^2 &= \sum_{\ell=1}^N \norm{\OQ_{h/m}^0
\OV_D z_h}_{L^2(\tau_\ell)}^2 \\
&\geq c(m)^2\sum_{i=1}^N \norm{\OV_D z_h}_{L^2(\tau_i)}^2
= c(m)^2\norm{\OV_D z_h}_{L^2(0,T)}^2.
\end{align*}
Considering the function $Z_h(t) = \int_0^t z_h(s)\, ds = 2(\OV z_h)(t)$, which
satisfies $Z_h\in S_{h;0,}^1(0,T)$, and using the self-duality of $L^2$, we can rewrite the
local relation \eqref{eq:auxiliary-local-relation} as
\begin{equation*}
c_{\tau_\ell}(m)\norm{Z_h}_{L^2(\tau_\ell)} \leq \norm{\OQ_{h/m}^0 Z_h}_{L^2(
\tau_\ell)} \leq \sup_{0\neq q_{h/m}\in S_{h/m}^0(\tau_\ell)} \frac{\int_{\tau_\ell} Z_h(t)
q_{h/m}(t)\, dt }{\norm{q_{h/m}}_{L^2(\tau_\ell)}}.
\end{equation*}
This is a local discrete inf-sup condition, and $c_{\tau_\ell}(m)$ is given as
the square root of the minimal $\lambda$ in the following generalized eigenvalue
problem
\begin{equation*}
(\OOM_{\tau_\ell,m}^{0,1})^\top (\OOM_{\tau_\ell,h/m}^{0,0})^{-1} \OOM_{
\tau_\ell,m}^{0,1} \underline{v} = \lambda \OOM_{\tau_\ell}^{1,1} \underline{v},
\end{equation*}
with matrices 
\begin{equation*}
\OOM_{\tau_\ell,m}^{0,1}[k,i] = \int_{\tau_\ell} \varphi_i^1(t)\varphi_k^{0,m}
(t)\,dt,\quad \OOM_{\tau_\ell,h/m}^{0,0}[r,k] = \frac{h_\ell}{m}\delta_{r,k},
\quad \OOM_{\tau_\ell}^{1,1}[i,j] = \int_0^T \varphi_j^1(t)\varphi_i^1(t)\, dt,
\end{equation*}
for $i,j=1,2$ and $k,r=1,\ldots,m$, where 
\begin{equation*}
\varphi_1^1(t)=\frac{t_\ell-t}{h_\ell},\quad \varphi_2^1(t)=\frac{t-t_{\ell-1}
}{h_\ell},\quad \varphi_k^{0,m}(t) = \begin{cases}
1, & t\in \left(t_{\ell-1} + \frac{h_\ell}{m}(k-1),
t_{\ell-1} + \frac{h_\ell}{m}k\right),\\
0, & \text{else}.
  \end{cases}
\end{equation*} 
By simple computations, we obtain
\begin{equation*}
\OOM_{\tau_\ell,m}^{0,1} = \frac{h_\ell}{2m^2}
\begin{pmatrix}
2m-1 & 1\\
2m-3 & 3\\
2m-5 & 5\\
\vdots & \vdots\\
3 & 2m-3\\
1 & 2m-1
\end{pmatrix},\quad \OOM_{\tau_\ell}^{1,1} = \frac{h_\ell}{6}\begin{pmatrix}
  2 & 1\\ 1 &2
\end{pmatrix}.
\end{equation*}
Then, using that
\begin{equation*}
\sum_{k=1}^m (2m-2k+1)^2 = \frac{m(4m^2-1)}{3},\quad \text{ and }\quad
\sum_{k=1}^m (2m-2k+1)(2k-1) = \frac{m(2m^2+1)}{3},
\end{equation*}
we get 
\begin{equation*}
(\OOM_{\tau_\ell,m}^{0,1})^\top (\OOM_{\tau_\ell,h/m}^{0,0})^{-1} \OOM_{\tau_\ell
,m}^{0,1} = \frac{h_\ell}{12m^2}\begin{pmatrix}
    4m^2-1 & 2m^2+1\\ 2m^2+1 & 4m^2-1
  \end{pmatrix}.
\end{equation*}
Moreover, we have 
\begin{align*}
(\OOM_{\tau_\ell}^{1,1})^{-1} (\OOM_{\tau_\ell,m}^{0,1})^\top (\OOM_{\tau_\ell,
h/m}^{0,0})^{-1} \OOM_{\tau_\ell,m}^{0,1} &= \frac{2}{h_\ell} \begin{pmatrix}
    2 & -1 \\ -1 & 2
  \end{pmatrix}\frac{h_\ell}{12 m^2}\begin{pmatrix}
    4m^2-1 & 2m^2+1\\ 2m^2+1 & 4m^2-1
  \end{pmatrix}\\
  & = \frac{1}{6m^2}\begin{pmatrix}
    6m^2-3 & 3\\ 3& 6m^2-3
  \end{pmatrix},
\end{align*}
for which the minimal eigenvalue is given by
\begin{equation*}
  c_{\tau_\ell}(m)^2 = \lambda_{\text{min},\tau_\ell} = 1-\frac{1}{m^2}. 
\end{equation*}
Since the eigenvalue does not depend on the element $\tau_\ell$, we have that
$c(m)^2 = 1-\frac{1}{m^2}$, which completes the proof.
}

The discrete inf-sup condition therefore holds uniformly for the choice $Y_h =
S^0(\Sigma_{h/m})$, $m\geq 2$, yielding a stable discrete formulation. However,
we have that $\dim(S^0(\Sigma_h)) \neq \dim(S^0(\Sigma_{h/m}))$,
which is inconvenient for numerical methods. As mentioned in
Remark~\ref{rmk:Equivalence}, we circumvent this difficulty by using that the
variational problem is equivalent to the saddle point formulation
\eqref{eq:SPF-practical} and therefore implement the latter. On the one hand,
this leads to a larger number of degrees of freedom. On the other hand, this
formulation also provides the additional variable $p_h$, which can be used as
an error estimator.

\subsubsection{Choose an integrated version of yourself: The stable pair
$S^0(\Sigma_h) - S_{0,}^1(\Sigma_h)$}\label{sssec:stab01}
For completeness, we conclude this section by reminding the reader that
Theorem~\ref{thm:YopV} showed that under the mesh Assumption~\ref{Assumption:mesh},
we get that the trial space $X_h = S^0(\Sigma_h)$ has the optimal test space
given by $Y_h^{\text{opt}} = S^1_{0,}(\Sigma_h)$. Hence, this combination also
constitutes a discrete inf-sup stable pair.

\section{Numerical experiments}
\label{sec:numerics}
In this Section we investigate the sharpness of our results. Since the instability
of the pair $S^0(\Sigma_h) - S^0(\Sigma_h)$ has already been reported in
\cite{Gua10}, we focus on the two stable pairs mentioned in
subsections~\ref{sssec:stable_LS} and \ref{sssec:stab01}.

Our code is implemented in \textsc{Python}. We use exact integration to obtain
the entries of the Galerkin matrices. We compute the discrete inf-sup constants
following the approach in \cite[Rem.~3.159]{John:2016}. Moreover, for this we
calculate the $L^2(\Sigma)$-norm as usual and the $\Dual{H^1_{,0}(\Sigma)}$-norm
following \cite[subsection~4.2]{HLS23}. For all experiments, we consider $L=3$,
while $T$ is specified for each case.

\subsection{The stable pair $S^0(\Sigma_h) - S_{0,}^1(\Sigma_h)$}
\label{ssec:results_Vpwc_pwl}
As proposed in Theorem~\ref{thm:YopV}, the following Galerkin discretisation is
expected to yield stable results: For $g_D \in L^2(\Sigma)$, find $z_h \in
S^0(\Sigma_h)\subset \Dual{H_{,0}^1(\Sigma)}$ such that
\begin{equation}\label{eq:Vpwc_pwl}
  \dual{\OV z_h}{v_h}_{L^2(\Sigma)} = \dual{g_D}{v_h}_{L^2(\Sigma)}, \quad
  \forall v_h \in S^1_{0,}(\Sigma_h).
\end{equation}

\begin{figure}[h]
\centering
\begin{subfigure}{0.475\textwidth}
\begin{tikzpicture}

\shade[left color=white,right color=lightgray] (0,0) rectangle (3,3) node[pos=.5] {$T_1$};
\shade[left color=lightgray,right color=white] (0,3) rectangle (3,6) node[pos=.5] {$T_2$};
\shade[left color=white,right color=lightgray] (0,6) rectangle (3,{2*pi}) node[pos=.5] {};

\draw[black] (0,0)--(0,{2*pi});
\draw[black] (3,0)--(3,{2*pi});

\def\leftTimes{ 0.00, 0.18, 0.55, 1.11, 1.66, 2.03, 2.22, 2.59,
    3.14, 3.33, 3.70, 4.25, 4.62, 4.80, 5.36, 5.73,{2*pi}}
\def\rightTimes{0.00, 0.37, 0.92, 1.11, 1.48, 2.03, 2.40, 2.59,
    3.14, 3.51, 3.70, 4.25, 4.62, 5.17, 5.36, 5.73, {2*pi}}

\foreach \t in \leftTimes {
    \draw (0,\t) node[circle, fill=black, scale=0.1, draw=black] {};
}

\foreach \t in \rightTimes {
    \draw (3,\t) node[circle, fill=black, scale=0.1, draw=black] {};
}

\draw[dotted] (-1.5,0)  -- (0,0) node[pos=-0.5] {$t=0$};
\draw[dotted] (-1.5,3)  -- (0,3) node[pos=-0.5] {$t=L$};
\draw[dotted] (-1.5,6)  -- (0,6) node[pos=-0.5] {$t=2L$};

\draw[dotted] (0,-1)  -- (0,0) node[pos=-0.3] {$\Sigma_h^0$};
\draw[dotted] (3,-1)  -- (3,0) node[pos=-0.3] {$\Sigma_h^L$};

\end{tikzpicture}
\caption{Case with $T\neq nL$ for
some $n\in \mathbb{N}$.}
\label{fig:initial_mesh}
\end{subfigure} \hfill
\begin{subfigure}{0.475\textwidth}
 \begin{tikzpicture}

\shade[left color=white,right color=lightgray] (0,0) rectangle (3,3) node[pos=.5] {$T_1$};
\shade[left color=lightgray,right color=white] (0,3) rectangle (3,6) node[pos=.5] {$T_2$};
\shade[left color=white,right color=lightgray] (0,6) rectangle (3,{2*pi}) node[pos=.5] {};

\draw[black] (0,0)--(0,6);
\draw[black] (3,0)--(3,6);
\draw[black, dashed] (0,6)--(0,{2*pi});
\draw[black, dashed] (3,6)--(3,{2*pi});

\def\leftTimes{0.00, 0.18, 0.55, 1.11, 1.66, 2.03, 2.22, 2.59,
    3, 3.18, 3.55, 4.11, 4.66, 5.03, 5.22, 5.59,6}
\def\rightTimes{0.00, 0.37, 0.92, 1.11, 1.48, 2.03, 2.40, 2.59,
    3,3.37, 3.92, 4.11, 4.48, 5.03, 5.40, 5.59,6}

\foreach \t in \leftTimes {
    \draw (0,\t) node[circle, fill=black, scale=0.1, draw=black] {};
}

\foreach \t in \rightTimes {
    \draw (3,\t) node[circle, fill=black, scale=0.1, draw=black] {};
}

\draw[dotted] (-1.5,0)  -- (0,0) node[pos=-0.5] {$t=0$};
\draw[dotted] (-1.5,3)  -- (0,3) node[pos=-0.5] {$t=L$};
\draw[dotted] (-1.5,6)  -- (0,6) node[pos=-0.5] {$t=2L$};

\draw[dotted] (0,-1)  -- (0,0) node[pos=-0.3] {$\Sigma_h^0$};
\draw[dotted] (3,-1)  -- (3,0) node[pos=-0.3] {$\Sigma_h^L$};

\end{tikzpicture}%
\caption{Case with $T=nL$ for
some $n\in \mathbb{N}$.}
\label{fig:initial_time_mesh}
\end{subfigure}
\caption{Initial non-uniform meshes that do not satisfy
Assumption~\ref{Assumption:mesh}.}
\end{figure}
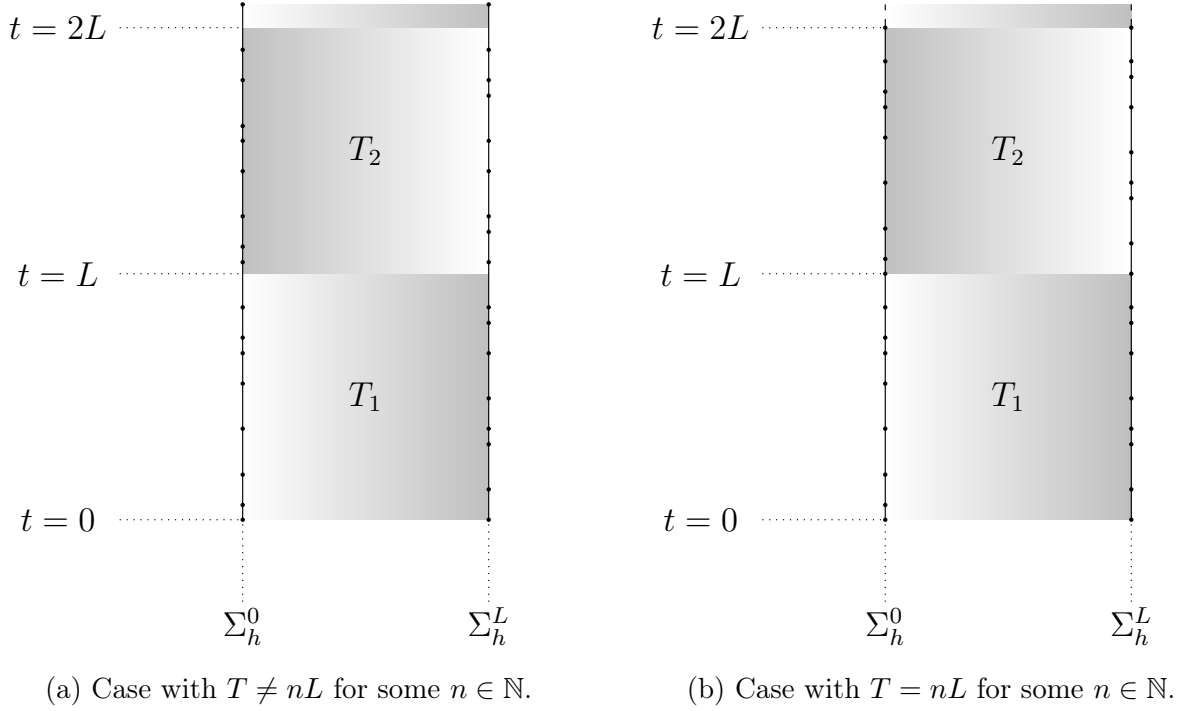

In the computation of the discrete inf-sup constant, we will also consider the
same uniform and non-uniform mesh constructions described earlier, considering
two cases: i) a uniform mesh with $T=nL$ and ii) a non-uniform mesh with
$T=n\pi$, based on the initial mesh depicted in Figure~\ref{fig:initial_mesh}.
Figure~\ref{fig:inf_sup_fixedT_V} displays the discrete inf-sup constant against
the number of mesh elements. There, we set the final time for the uniform case to
$T=6$, and in the non-uniform case to $T=2\pi$, and increase the level of
refinement. It becomes clear from Figure~\ref{fig:inf_sup_fixedT_V} that for
fixed $T$, the inf-sup constant does not asymptotically depend on the mesh
parameter $h$, which confirms discrete stability of \eqref{eq:Vpwc_pwl}.

Next, we study numerically how the discrete inf-sup constant depends on the final
time $T$. To make a fair comparison, we fix the final time of both the uniform
and non-uniform meshes to be $T=nL$. In the non-uniform case, the family of meshes
corresponding to different final times $T$ is constructed by truncating the
initial mesh depicted in Figure~\ref{fig:initial_mesh} at $T=3$, and
extending it to $T>3$ by repeatedly appending time-translated copies of this
truncated mesh structure, as shown in Figure~\ref{fig:initial_time_mesh}.

As shown in Figure~\ref{fig:inf_sup_timeslices_V} the discrete inf-sup constant
decreases when the terminal time $T$ increases, yet it remains strictly positive.
In the uniform case the discrete inf-sup constant seems to coincide exactly with
the inf-sup constant $\alpha_{\text{is}}^{\OV}(n)$ from \eqref{eq:inf-sup-optimal},
as expected. Surprisingly, the non-uniform case also closely resembles $\alpha_{
\text{is}}^{\OV}(n)$, which gives us a hint on which direction to go if we want
to generalize the results from Section~\ref{sec:mainResults}.

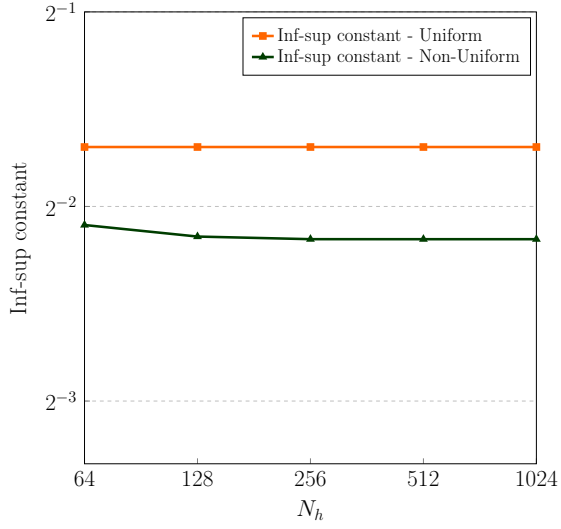
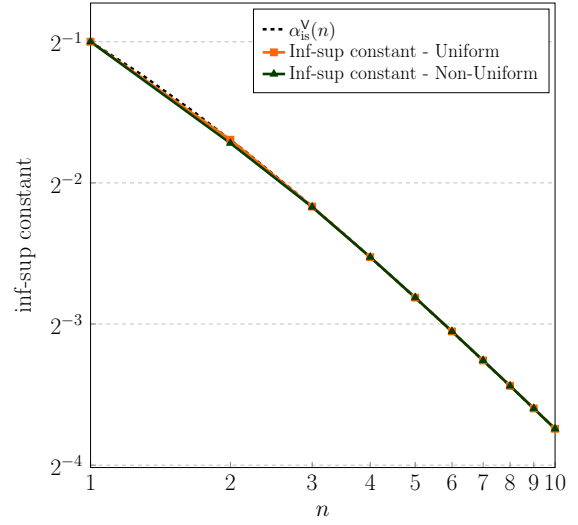
\begin{figure}[h]
\centering
\begin{subfigure}{0.475\textwidth}
 \pgfplotstableread{Data/inf_sup_1stkind_dualH_Tfixed_uniform.txt}\infsupUni
\pgfplotstableread{Data/inf_sup_1stkind_dualH_Tfixed_terriblemesh.txt}\infsupNon



\pgfplotstableset{
    create on use/Nh/.style={
        create col/copy column from table={Data/N.txt}{N}}
    }
{\scalefont{1.5}
\begin{tikzpicture}
\begin{loglogaxis}[
    log basis y ={2},
    log basis x = {2},
    width=14cm,
    height=14cm,
    xlabel={$N_h$},
    ylabel={Inf-sup constant},
    xmin=64, xmax=1024,
    ymin = 0.1, ymax=0.5,
    ytick pos= left,
    xtick pos= bottom,
    xtick={64,128,256,512,1024},
    xticklabels= {64,128,256,512,1024},
    legend pos=north east,
    legend cell align={left},
    legend style ={font = \large, inner sep= 6pt, outer sep = -6pt},
    ymajorgrids=true,
    grid style=dashed,
]






\addplot[orange,mark= square*,line width=2pt, mark options={solid}] table [x = Nh, y index = 0] {\infsupUni};
\addlegendentry{Inf-sup constant - Uniform}

\addplot[green!60!black, mark=triangle*,mark options=solid, line width=2pt] table [x = Nh, y index = 0] {\infsupNon};
\addlegendentry{Inf-sup constant - Non-Uniform}

\end{loglogaxis}
\end{tikzpicture}
}%
\caption{Results for different refinement levels and two families of meshes:
the uniform case with $T=6$; and the non-uniform case with $T=2\pi$ with
initial mesh as in Figure~\ref{fig:initial_mesh}.}
\label{fig:inf_sup_fixedT_V}
\end{subfigure}\hfill
\begin{subfigure}{0.475\textwidth}
 \pgfplotstableread{Data/inf_sup_1stkind_dualH_Tdiff_uniform.txt}\infsupUni
\pgfplotstableread{Data/inf_sup_1stkind_dualH_Tdiff_terriblemesh.txt}\infsupNon


{\scalefont{1.5}
\begin{tikzpicture}
\begin{loglogaxis}[
    log basis y ={2},
    log basis x = {2},
    width=14cm,
    height=14cm,
    xlabel={$n$},
    ylabel={inf-sup constant},
    xmin=1, xmax=10,
    ytick pos= left,
    xtick pos= bottom,
    xtick={1,2,3,4,5,6,7,8,9,10},
    xticklabels= {1,2,3,4,5,6,7,8,9,10},
    ytick = {0.5,0.25,0.125, 0.0625},
    legend pos=north east,
    legend cell align={left},
    legend style ={font = \large, inner sep= 6pt, outer sep = -6pt},
    ymajorgrids=true,
    grid style=dashed,
]


\addplot [
    domain= 1:10, 
    samples=10, 
    color=black,
    dashed,
    line width=2pt
]
{sin(deg(pi/(4*x+2)))};
\addlegendentry{$\alpha_{\text{is}}^{\OV}(n)$}




\addplot[orange,mark= square*,line width=2pt, mark options={solid}] table [x expr=\coordindex+1, y index = 0] {\infsupUni};
\addlegendentry{Inf-sup constant - Uniform}

\addplot[green!60!black, mark=triangle*,mark options=solid, line width=2pt] table [x expr = \coordindex+1, y index = 0] {\infsupNon};
\addlegendentry{Inf-sup constant - Non-Uniform}

\end{loglogaxis}
\end{tikzpicture}
}%
\caption{Results for different final times and two meshes: a uniform mesh with
$T=nL$ and $h_0=h_L=3/16$; a non-uniform one depicted in
Figure~\ref{fig:initial_time_mesh}.}
\label{fig:inf_sup_timeslices_V}
\end{subfigure}
\caption{Discrete inf-sup constants for $S^0(\Sigma_h) - S_{0,}^1(\Sigma_h)$.}
\end{figure}

\subsection{The stable pair $S^0(\Sigma_h) - S^0(\Sigma_{h/2})$}
In Section~\ref{sssec:stable_LS}, it is shown that the following Galerkin
discretisation is expected to yield stable results: For $g_D \in L^2(\Sigma)$,
find $z_h \in S^0(\Sigma_h)\subset \Dual{H_{,0}^1(\Sigma)}$ such that
\begin{equation}\label{eq:V_LS_m2}
\dual{\OV z_h}{v_{h/2}}_{L^2(\Sigma)} = \dual{g_D}{v_{h/2}}_{L^2(\Sigma)},
\quad \forall v_{h/2} \in S^0(\Sigma_{h/2}).
\end{equation}

We again consider the discrete inf-sup constant both for fixed ending time
and on multiple time-slices. For fixed ending time we consider the same
meshes as in Section~\ref{ssec:results_Vpwc_pwl} and also an adaptive mesh
using the error estimator from Theorem~\ref{thm:errorEstimator}. The results
are presented in Figure~\ref{fig:inf_sup_fixedT_LS}. Recall that stability is
only shown on meshes satisfying Assumption~\ref{Assumption:mesh}, this explains
the observed decay of the discrete inf-sup constant in the case of the
non-uniform mesh in Figure~\ref{fig:inf_sup_fixedT_LS}. When computing the
discrete inf-sup constant for different ending times, only uniform refinements
are considered to ensure 
a fair comparison among meshes. The results are presented
in Figure~\ref{fig:inf_sup_timeslices_LS}.

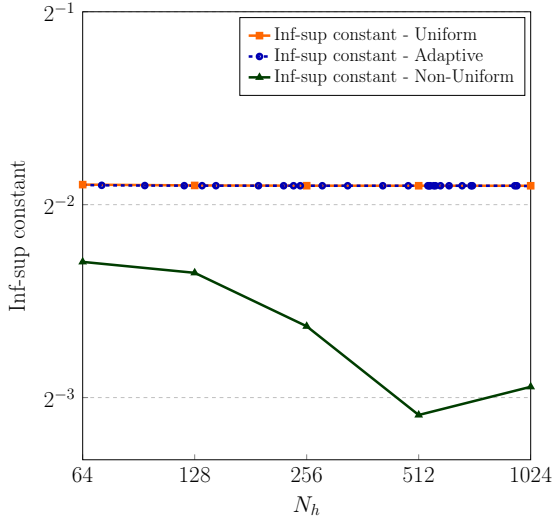
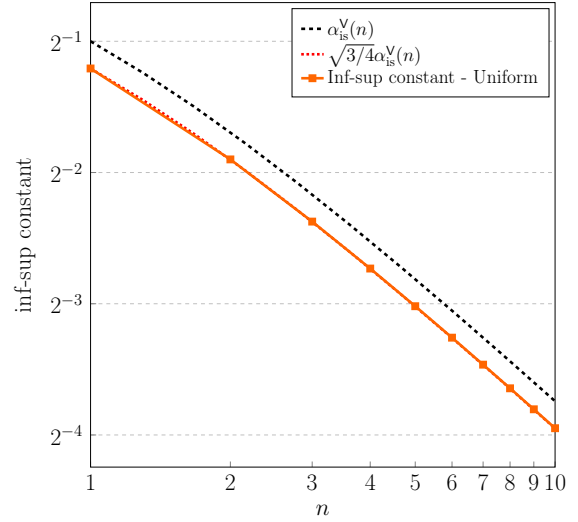
\begin{figure}[H]
\centering
\begin{subfigure}{0.475\textwidth}
\pgfplotstableread{Data/inf_sup_1stkind_dualH_Tfixed_uniform_LS.txt}\infsupUni
\pgfplotstableread{Data/inf_sup_1stkind_dualH_Tfixed_adaptive_LS.txt}\infsupAdap
\pgfplotstableread{Data/inf_sup_1stkind_dualH_Tfixed_terriblemesh_LS.txt}\infsupNon



\pgfplotstableset{
    create on use/Nh/.style={
        create col/copy column from table={Data/N.txt}{N}}
    }
\pgfplotstableset{
create on use/Nh2/.style={
    create col/copy column from table={Data/N_adaptive_m2.txt}{N}}
}
{\scalefont{1.5}
\begin{tikzpicture}
\begin{loglogaxis}[
    log basis y ={2},
    log basis x = {2},
    width=14cm,
    height=14cm,
    xlabel={$N_h$},
    ylabel={Inf-sup constant},
    xmin=64, xmax=1024,
    ymin = 0.1, ymax=0.5,
    ytick pos= left,
    xtick pos= bottom,
    xtick={64,128,256,512,1024},
    xticklabels= {64,128,256,512,1024},
    legend pos=north east,
    legend cell align={left},
    legend style ={font = \large, inner sep= 6pt, outer sep = -6pt},
    ymajorgrids=true,
    grid style=dashed,
]






\addplot[orange,mark= square*,line width=2pt, mark options={solid}] table [x = Nh, y index = 0] {\infsupUni};
\addlegendentry{Inf-sup constant - Uniform}

\addplot[blue!70!black,mark=o,line width=2pt, mark options={solid}, dashed] table [x = Nh2, y index = 0] {\infsupAdap};
\addlegendentry{Inf-sup constant - Adaptive}

\addplot[green!60!black, mark=triangle*,mark options=solid, line width=2pt] table [x = Nh, y index = 0] {\infsupNon};
\addlegendentry{Inf-sup constant - Non-Uniform}

\end{loglogaxis}
\end{tikzpicture}
}%
\caption{Results for different refinement levels and three families of meshes:
uniform and adaptive meshes with $T=6$; and the non-uniform case with $T=2\pi$ and
initial mesh as in Figure~\ref{fig:initial_mesh}.}
\label{fig:inf_sup_fixedT_LS}
\end{subfigure}\hfill
\begin{subfigure}{0.475\textwidth}
 \includestandalone[width=\textwidth]{Tikz_pics/inf_sup_Tdiff_LS}%
\caption{Results for different final times and a uniform mesh with
$T=nL$ and $h_0=h_L=3/16$.\\[0.45cm]}
\label{fig:inf_sup_timeslices_LS}
\end{subfigure}
\caption{Discrete inf-sup constants for $S^0(\Sigma_h) - S^0(\Sigma_{h/2})$.}
\end{figure}

\section{Stability With Weaker Mesh Constraints}
\label{sec:NoConstraints}
The fact that our numerical experiments gave evidence of numerical stability
when violating our mesh constraint raises the question of whether stable pairings
can still be constructed without imposing this assumption. Indeed, it has been
shown in some cases that the formulation becomes elliptic when composing the
weakly singular operator $\OV$ with a suitable operator \cite{ADG09,SUZ21},
and therefore stability can be obtained without imposing any restriction on the
mesh.

In order to derive a result that does not require mesh assumptions, we
first take a diversion and consider the double layer potential $\OK:H^s_{0,}(
\Sigma)\to H^s_{0,}(\Sigma)$ defined as \cite{HU:2025}
\begin{equation*}
(\OK z)(x,t):=\frac{1}{2}\int_\Sigma \partial_{n_y}\operatorname{H}(x-y,t-\tau)
z(y,\tau) \, ds_y\, d\tau = \begin{cases}
    -\frac{1}{2}z(L,t-L), & x=0 \text{ and } t\geq L,\\
    -\frac{1}{2}z(0,t-L), & x=L \text{ and } t\geq L,\\
    0 & t<L.
  \end{cases} 
\end{equation*} 
For the purpose at hand, the following relation will be useful:
\begin{lemma}\label{lem:relation V and K}
It holds that
\begin{equation*}
\OV\partial_t = -\left(-\frac{1}{2}\OI + \OK\right):H^s_{0,}(\Sigma) \to H^s_{
0,}(\Sigma), \quad s\in[0,1],
\end{equation*}
where $\partial_t$ is understood in the distributional sense.
\end{lemma}
\proof{
We begin by recalling that $\partial_t:H^s_{0,}(\Sigma)\to [H^{1-s}_{,0}(\Sigma
)]'$ and $\OV:[H^{1-s}_{,0}(\Sigma)]'\to H^{s}_{0,}(\Sigma)$ are isomorphisms
for $s\in[0,1]$ and thus the concatenation is well defined.
Further, by the definition of the operator, we compute
\begin{equation*}
(\OV \partial_t z) (x,t) = \left\{\begin{array}{cc}
\frac{1}{2}z(0,t)   + \frac{1}{2}z(L,t-L), & x=0\text{ and } t\geq L,\\[0.03in]
\frac{1}{2}z(0,t-L) + \frac{1}{2}z(L,t),   & x=L\text{ and } t\geq L,\\
    0 & t<L.
\end{array}\right\} = \frac{1}{2}z(t) - (\OK z)(t).
\end{equation*}
}
As shown in \cite[Theorem 3.2]{HU:2025}, the following ellipticity estimate
holds true
\begin{equation}\label{eq:ellipticity K}
\left|\left\langle\left(-\frac{1}{2}\OI+\OK\right)w,w\right\rangle_{\Sigma
}\right| \geq \alpha_c^{\OK}(n)\norm{w}_{L^2(\Sigma)}^2,
\end{equation}
where $\alpha_c^{\OK}(n):= \sin\left(\frac{\pi}{2(n+1)}\right)^2$, which
is the starting point for our first attempt to show discrete stability without
the mesh assumption \ref{Assumption:mesh}.

\subsection{A dull stability estimate without mesh constraints}
\label{ssec:fistAttempt}
\begin{theorem}\label{thm:discrete_inf-sup_energetic-BEM}
Let $n\in \mathbb{N}$ denote the number of time slices $T_j$ as defined in
\eqref{eq:timeslices}. There holds the discrete inf-sup condition
\begin{align}
\alpha_c^{\OK}(n) \leq \inf_{0\neq z_h\in S^0(\Sigma_h)} 
\sup_{0\neq q_h \in S_{0,}^1(\Sigma_h)} \frac{\abs{\dual{ \OV z_h}{ q_h}
_{\Sigma} }}{\norm{z_h}_{[H_{,0}^1(\Sigma)]'}\norm{q_h}_{L^2(\Sigma)}}.
\end{align}
\end{theorem}
\proof{ 
Let $z_h\in S^0(\Sigma_h)$ be arbitrary but fixed. Consider $Z_h(t) = \int_0^t
z_h(s)\, ds \in S_{0,}^1(\Sigma_h)$ and recall from \eqref{eq:H-1 to L2 norm}
that $\norm{z_h}_{[H^1_{,0}(\Sigma)]'} = \norm{Z_h}_{L^2(\Sigma)}$. Then, using
the relation in Lemma \ref{lem:relation V and K} and the ellipticity estimate
\eqref{eq:ellipticity K}, we derive
\begin{align*}
\sup_{0\neq q_h\in S_{0,}^1(\Sigma_h)}\frac{|\dual{\OV z_h}{q_h}_{\Sigma}|}{
\norm{q_h}_{L^2(\Sigma)}}
& \geq \frac{|\dual{\OV \partial_t Z_h}{Z_h}_{\Sigma}|}{\norm{Z_h}_{L^2(\Sigma)}}
\geq \alpha_c^{\OK}(n)\norm{Z_h}_{L^2(\Sigma)}
= \alpha_c^{\OK}(n) \norm{z_h}_{[H^1_{,0}(\Sigma)]'}.
\end{align*} 
}

This shows that stability also holds for arbitrary meshes. However, the
corresponding stability constant $\alpha_c^{\OK}(n)$ exhibits a qualitatively
different behavior compared to the constants in \eqref{eq:inf-sup-optimal} and
Theorem \ref{thm:discrete inf sup m}, which depend on $\alpha_{\mathrm{is}}^{\OV
}(n)$. More precisely, we see that
\[
\alpha_c^{\OK}(n)\sim n^{-2},
\qquad \text{ and } \qquad \alpha_{\mathrm{is}}^{\OV}(n)\sim n^{-1},
\qquad \text{ when } n\to\infty.
\]
As displayed in Figure \ref{fig:inf_sup_timeslices_V}, our numerical
experiments indicate that the computed discrete inf-sup constant follows the
improved behavior of $\alpha_{\mathrm{is}}^{\OV}(n)\sim n^{-1}$ even on arbitrary
meshes. Hence, our estimate in Theorem~\ref{thm:discrete_inf-sup_energetic-BEM}
is not sharp enough. The remainder of this section is dedicated to improving this
result.

\subsection{Choose someone who lives up to their projection: \\stability under
a different mesh condition}
\label{ssec:victory}
Now we finally move on to obtaining the desired improved inf-sup constant, which
asymptotically behaves like the numerically observed power in $n$. For this, let us
introduce the $L^2$-projection onto the piecewise linear functions $\OQ_h^1:L^2
(\Sigma) \to S_{0,}^1(\Sigma_h)$ defined as
\begin{equation*}
\dual{\OQ_h^1 v}{q_h}_{\Sigma} = \dual{v}{q_h}_{\Sigma},\quad \forall q_h\in
  S_{0,}^1(\Sigma_h),
\end{equation*}
and note that, for $z_h \in S^0(\Sigma_h)$, we obtain the lower bound
\begin{align*}
\sup_{0\neq q_h\in S_{0,}^1(\Sigma_h)} \frac{\dual{\OV z_h}{q_h}_{\Sigma}}
{\norm{q_h}_{L^2(\Sigma)}}  &
= \sup_{0\neq q_h\in S_{0,}^1(\Sigma_h)} \frac{\dual{\OQ_h^1\OV z_h}{q_h}_{\Sigma
}}{\norm{q_h}_{L^2(\Sigma)}} \\
&\geq \frac{\dual{\OQ_h^1 \OV z_h}{\OQ_h^1 \OV z_h}_{\Sigma}}{\norm{\OQ_h^1\OV
z_h}_{L^2(\Sigma)}} = \norm{\OQ_h^1 \OV z_h}_{L^2(\Sigma)};
\end{align*}
while, using the Cauchy--Schwarz inequality, we get the upper bound
\begin{align*}
\sup_{0\neq q_h\in S_{0,}^1(\Sigma_h)} \frac{\dual{\OV z_h}{q_h}_{\Sigma}}{\norm{
q_h}_{L^2(\Sigma)}}  
&= \sup_{0\neq q_h\in S_{0,}^1(\Sigma_h)} \frac{\dual{\OQ_h^1\OV z_h}{q_h}_{
\Sigma}}{\norm{q_h}_{L^2(\Sigma)}}\\
&\leq \sup_{0\neq q_h\in S_{0,}^1(\Sigma_h)} \frac{\norm{\OQ_h^1\OV z_h}_{L^2(
\Sigma)}\norm{q_h}_{L^2(\Sigma)}}{\norm{q_h}_{L^2(\Sigma)}}
  = \norm{\OQ_h^1 \OV z_h}_{L^2(\Sigma)}.
\end{align*}
Therefore, altogether,
\begin{equation}\label{eq:inf-sup-arbitrary-mesh-exact representation}
\sup_{0\neq q_h\in S_{0,}^1(\Sigma_h)}\frac{\dual{\OV z_h}{q_h}_{\Sigma}}{
\norm{q_h}_{L^2(\Sigma)}} = \norm{\OQ_h^1 \OV z_h}_{L^2(\Sigma)}.
\end{equation}
Consequently, proving the desired discrete inf-sup estimate boils down to showing
that
\begin{equation*}
\norm{\OQ_h^1 \OV z_h}_{L^2(\Sigma)} \geq \widetilde{\alpha}_{\text{is}}^{\OV
}(n) \norm{z_h}_{[H^1_{,0}(\Sigma)]'},
\end{equation*}
for a positive constant $\widetilde{\alpha}_{\text{is}}^{\OV}(n) \sim n^{-1}$. 
Introducing $Z_h(t) = \int_0^t z_h(s)\, ds \in S_{0,}^1(\Sigma_h)$, and
using the established identity in Lemma \ref{lem:relation V and K}, we have that
\begin{align*}
  \norm{\OQ_h^1 \OV z_h}_{L^2(\Sigma)}
  &=
  \norm{\OQ_h^1 \OV \partial_t Z_h}_{L^2(\Sigma)}
  =
  \left\|\OQ_h^1\left(\frac{1}{2}\OI-\OK\right)Z_h\right\|_{L^2(\Sigma)}.
\end{align*}
Furthermore, noting that $\OQ_h^1 Z_h = Z_h$, and defining $\OS := 2\OK:L^2(
\Sigma)\to L^2(\Sigma)$ and $\OS_h := \OQ_h^1 \OS:S_{0,}^1(\Sigma_h)\to S^1_{0,
}(\Sigma_h)$,
this becomes
\begin{equation*}
  \norm{\OQ_h^1 \OV z_h}_{L^2(\Sigma)}
  =  \norm{\frac{1}{2}(\OI-\OQ_h^1 S)Z_h}_{L^2(\Sigma)}
  =  \frac{1}{2}\norm{(\OI-\OS_h)Z_h}_{L^2(\Sigma)}.
\end{equation*}
Recalling that $\norm{z_h}_{[H^1_{,0}(\Sigma)]'} = \norm{Z_h}_{L^2(\Sigma)}$,
see \eqref{eq:H-1 to L2 norm}, we thus need an estimate of the form
\begin{equation*}
\frac{1}{2}\norm{(\OI-\OS_h)Z_h}_{L^2(\Sigma)}\geq\widetilde{\alpha}_{\text{is}
}^{\OV}(n)\norm{Z_h}_{L^2(\Sigma)}.
\end{equation*}
  A first step is taken in the next Lemma. 
\begin{lemma}\label{lem:asymptotic bound an alpha_is using norm S_h^n}
Let $n\in \mathbb{N}$ denote the number of time slices $T_j$ as defined in
\eqref{eq:timeslices}. Let $\OS_h^n:= (\OS_h)^n$ denote the n-th power of
$\OS_h$. Then
\begin{equation*}
  \widetilde{\alpha}_{\text{is}}^{\OV}(n)\geq \frac{1-\norm{\OS_h^n}_{L^2(
  \Sigma) \to L^2(\Sigma)}}{2n}.
\end{equation*}
\end{lemma}
\proof{
First, note that since the $L^2$-projection is bounded with constant one and
$\OS$ is a shift operator, we have that $\norm{\OS_h}_{L^2(\Sigma)\to L^2(
\Sigma)} = \norm{\OQ_h^1 \OS}_{L^2(\Sigma)\to L^2(\Sigma)} \leq 1$.
Consequently, we also get $\norm{\OS_h^k}_{L^2(\Sigma)\to L^2(\Sigma)} \leq 1$
for all $k=0,\ldots,n$. Now, using the identity
\begin{equation*}
  (\OI-\OS_h^n) = \left(\sum_{k=0}^{n-1}\OS_h^k\right)(\OI-\OS_h),
\end{equation*}
we get
\begin{equation*}
  \norm{(\OI-\OS_h^n)Z_h}_{L^2(\Sigma)}
  \leq
  \sum_{k=0}^{n-1}\norm{\OS_h^k}_{L^2(\Sigma)\to L^2(\Sigma)}\norm{(\OI-\OS_h)
  Z_h}_{L^2(\Sigma)}
  \leq
  n \norm{(\OI-\OS_h)Z_h}_{L^2(\Sigma)}.
\end{equation*}
On the other hand, using the inverse triangle inequality gives
\begin{equation*}
  \norm{(\OI-\OS_h^n)Z_h}_{L^2(\Sigma)}
  \geq
  (1-\norm{\OS_h^n}_{L^2(\Sigma) \to L^2(\Sigma)})\norm{Z_h}_{L^2(\Sigma)},
\end{equation*}
which finishes the proof.
}

Hence, the asymptotic order of decay depends on the quantity $\norm{
\OS_h^n}_{L^2(\Sigma)\to L^2(\Sigma)}$. Noting that $\OS$ is a right shift, we
have that $\OS^nZ_h=0$. Thus, if for any $Z_h \in S_{0,}^1(\Sigma_h)$, it holds
that $\OS_h Z_h = \OS Z_h$, we get $\widetilde{\alpha}_{\text{is}}^{\OV}(n)
\geq\frac{1}{2n} \simeq \alpha_{\text{is}}^{\OV}(n)$.
For general meshes though, $\OS_h Z_h \neq \OS Z_h$. More precisely,
using that
\begin{equation}\label{eq:auxiliary S_h and S relation}
  \OS_h^k -\OS^k 
  =
  \sum_{j=0}^{k-1}\left(\OS_h^{k-j}\OS^j-\OS_h^{k-1-j}\OS^{j+1}\right)
  =
  \sum_{j=0}^{k-1}\OS_h^{k-1-j}(\OS_h-\OS)\OS^j,
\end{equation} 
we obtain
\begin{align*}
  \OS_h^kZ_h 
  &= \OS^k Z_h -\sum_{j=0}^{k-1} \OS_h^{k-1-j}(I-\OQ_h^1)\OS^{j+1}Z_h\\
  &= \OS^k Z_h + \text{propagation of projection errors.}  
\end{align*}
This implies that if the mesh does not resolve the propagated function well
enough, the discrete stability will get worse. Reversing the roles of $\OS_h$
and $\OS$ in \eqref{eq:auxiliary S_h and S relation}, we also obtain the
representation
\begin{equation}\label{eq:S_h representation in S}
  \OS_h^kZ_h
  = \OS^k Z_h - \sum_{j=0}^{k-1} \OS^{k-1-j}(I-\OQ_h^1) \OS \OS_h^j Z_h,
\end{equation}
that will be used for proving a worst case estimate that shows that
asymptotically we are at least as good, as the coercivity estimate in
Theorem~\ref{thm:discrete_inf-sup_energetic-BEM}.
\begin{lemma}\label{lem:worst case estimate norm S_h^n}
For any given mesh and any $n\in \mathbb{N}$ satisfying $n\geq T/L$ it holds
that
  \begin{equation*}
    \norm{\OS_h^n}_{L^2(\Sigma)\to L^2(\Sigma)}
    \leq 
    \sqrt{\frac{n}{n+1}}. 
  \end{equation*}
\end{lemma}
\proof{
Since $\OS^n Z_h =0$ for all $Z_h \in S_{0,}^1(\Sigma_h)$, and $\norm{\OS}_{L^2
(\Sigma)\to L^2(\Sigma)}\leq 1$, we have by \eqref{eq:S_h representation in S}
that
\begin{align}\label{eq:boundbysum}
  \norm{\OS_h^n Z_h}_{L^2(\Sigma)}^2
  &=
  \norm{\sum_{j=0}^{n-1}\OS^{n-1-j}(\OI-\OQ_h^1)\OS\OS_h^j Z_h}_{L^2(\Sigma)}^2
  \nonumber\\ &\leq
  \left(\sum_{j=0}^{n-1} \norm{\OS^{n-1-j}}_{L^2(\Sigma)\to L^2(\Sigma)} \norm{
  (\OI-\OQ_h^1) \OS \OS_h^j Z_h}_{L^2(\Sigma)}\right)^2\nonumber\\
  & \leq
  \left(\sum_{j=0}^{n-1}\norm{(\OI-\OQ_h^1)\OS\OS_h^j Z_h}_{L^2(\Sigma)}\right)^2
  \nonumber\\& \leq
  n \sum_{j=0}^{n-1}\norm{(\OI-\OQ_h^1) \OS \OS_h^j Z_h}_{L^2(\Sigma)}^2,
\end{align}
where the last step follows from using the finite dimensional Cauchy--Schwarz
inequality. Next, the orthogonality of the $L^2$-projection gives us that
\begin{align*}
  \norm{(\OI-\OQ_h^1) \OS \OS_h^j Z_h}_{L^2(\Sigma)}^2
  &=
  \norm{\OS \OS_h^j Z_h}_{L^2(\Sigma)}^2 - \norm{\OQ_h^1 \OS \OS_h^j Z_h}_{L^2(
  \Sigma)}^2\\
  &=
  \norm{\OS \OS_h^j Z_h}_{L^2(\Sigma)}^2 - \norm{\OS_h^{j+1} Z_h}_{L^2(\Sigma)}^2\\
  & \leq
  \norm{\OS_h^j Z_h}_{L^2(\Sigma)}^2 - \norm{\OS_h^{j+1} Z_h}_{L^2(\Sigma)}^2.
\end{align*}
Hence, the last sum in \eqref{eq:boundbysum} can be estimated by a telescoping
sum and we obtain
\begin{align*}
\norm{\OS_h^n Z_h}_{L^2(\Sigma)}^2
& \leq
n \sum_{j=0}^{n-1}\norm{(\OI-\OQ_h^1) \OS \OS_h^j Z_h}_{L^2(\Sigma)}^2\\
& \leq
n \sum_{j=0}^{n-1} \left(\norm{\OS_h^j Z_h}_{L^2(\Sigma)}^2 - \norm{\OS_h^{j+1}
Z_h}_{L^2(\Sigma)}^2\right)\\
& = n(\norm{Z_h}_{L^2(\Sigma)}^2 - \norm{\OS_h^{n} Z_h}^2),
\end{align*}
which, after reordering, gives the desired result.
}
\begin{remark}
Let $n\in \mathbb{N}$ denote the number of time slices. Combining
Lemma~\ref{lem:asymptotic bound an alpha_is using norm S_h^n} with
Lemma~\ref{lem:worst case estimate norm S_h^n}, we have that
\begin{equation*}
  \widetilde{\alpha}_{\text{is}}^{\OV}(n) \geq \frac{1}{2n\sqrt{n+1}(\sqrt{n+1}
  +\sqrt{n})} \simeq \frac{1}{4n^2} \simeq \alpha_{\text{c}}^{\OK}(n).
\end{equation*}
\end{remark}
\begin{remark}
If Assumption \ref{Assumption:mesh} on the mesh holds true, we have that
\begin{equation*}
  \OS_h Z_h = \OQ_{h}^1 \OS Z_h = \OS Z_h,\quad \forall Z_h\in S_{0,}^1(\Sigma_h). 
\end{equation*}
In particular, it holds that $\norm{\OS_h^n}_{L^2(\Sigma)\to L^2(\Sigma)} =
\norm{\OS^n}_{L^2(\Sigma)\to L^2(\Sigma)} =0$ and thus
\begin{equation*}
\widetilde{\alpha}_{\text{is}}^{\OV}(n) \geq \frac{1}{2n}\simeq \alpha_{
\text{is}}^{\OV}(n).
\end{equation*}
This indicates that if the mesh is able to resolve the propagated information
sufficiently well, the inf-sup constant behaves like
$\widetilde{\alpha}_{\text{is}}^{\OV}\sim n^{-1}$.
\end{remark}

\subsubsection{Approximated compatibility}
\label{sssec:appComp}
In the following, we will give a sufficient condition on the mesh that guarantees
the desired asymptotic scaling. We assume $T>L$, as otherwise no mesh condition
is required.

Let $\mathcal{N}_{h,\ast}.= \{t_{i,\ast}\}_{i=0}^{N_\ast}$ denote the nodes of
$\Sigma_{h}^\ast$, $\ast \in \{0,L\}$ in the temporal direction. Moreover, we
introduce the shifted set of nodes
\begin{equation*}
\widehat{\mathcal{N}}_{h,\ast}   =
\{\widehat t_{i,\ast}\}_{i=0}^{\widehat N_\ast} := \{0,T\} \cup \{t_{i,\ast}+L
:\, t_{i,\ast}\in \mathcal{N}_{h,\ast}, \, t_{i,\ast}+L<T\},
\end{equation*}
consisting of all shifted nodes of $\mathcal{N}_{h,\ast}$ that are within the
interval $[0,T]$ and the nodes $0$ and $T$. Clearly, $\widehat N_\ast \leq
N_\ast+1$.

We introduce the index set 
\begin{equation*}
J_{\ell,0} := \big\{ i\in\{1,\ldots,\widehat N_0-1\}: \, \widehat t_{i,0} \in
(t_{\ell-1,L},t_{\ell,L})\big\},\,  \ell=1,\ldots,N_L,
\end{equation*} 
containing the indices of all shifted nodes from the left boundary that lie
strictly within an element of the right boundary. Analogously, we define the
index set
\begin{equation*}
J_{\ell,L} :=\big\{ i\in\{1,\ldots,\widehat N_L-1\}:\, \widehat t_{i,L} \in
(t_{\ell-1,0},t_{\ell,0})\big\}, \, \ell = 1 ,\ldots, N_0.
\end{equation*} 
We pose the following assumptions. 
\begin{assumption}\label{ass:mesh weaker}
There exists a constant $M_m\in\mathbb N$, independent of $h$ and
$n$, such that
\begin{equation*}
  |J_{\ell,0}|\leq M_m,
  \quad \ell=1,\ldots,N_L,
  \quad
  |J_{k,L}|\leq M_m,
  \quad k=1,\ldots,N_0.
\end{equation*}
Moreover, let $h_{\ell,\ast} := t_{\ell,\ast}-t_{\ell-1,\ast}$ and $\whh_\ell:=
\widehat t_{\ell,\ast} - \widehat t_{\ell-1,\ast}$, $\ast \in \{0,L\}$. Then,
we assume
\[
  h_{\ell,L} \leq M_m\min\{\whh_{i,0},\whh_{i+1,0}\},
  \quad\forall i\in J_{\ell,0},
\]
and
\[
  h_{\ell,0} \leq M_m\min\{\whh_{i,L},\whh_{i+1,L}\},
  \quad \forall i\in J_{\ell,L}.
\]
\end{assumption}
\noindent 
The above assumption implies that the number of shifted nodes lying within an
element is uniformly bounded, and that the target elements containing the
shifted nodes are not arbitrarily large.

Now, in order to state our results, let us define the minimal temporal
distance of a shifted node $\widehat t_{i,0}\in \widehat{\mathcal{N}}_{h,0}$
from the left boundary to the (original) nodes on the right boundary
\begin{equation*}
d_{i,0}:=\min_{t_{\ell,L} \in\mathcal{N}_{h,L}} |\widehat t_{i,0}- t_{\ell,L}|.
\end{equation*}
Analogously, for $\widehat t_{i,L} \in \widehat{\mathcal{N}}_{h,L}$, we consider
\begin{equation*}
d_{i,L}:=\min_{t_{\ell,0} \in\mathcal{N}_{h,0}} |\widehat t_{i,L} -t_{\ell,0}|,
\end{equation*}
and also introduce the quantities
\begin{equation*}
\eta_{h,\ast}:= \max_{i=1,\ldots,\widehat N_\ast-1} \frac{d_{i,\ast}}{\min\{
\whh_{i,\ast},\whh_{i+1,\ast}\}},\quad \ast \in\{0,L\},
\end{equation*}
as well as 
\begin{equation*}
  \eta_h := \max\{\eta_{h,0},\eta_{h,L}\}. 
\end{equation*}
Then we have the following result. 
\begin{theorem}
\label{thm:EstimatesWeakerMesh}
Let Assumption \ref{ass:mesh weaker} hold true. Then
\begin{equation}\label{eq:mesh and n dependent S_h bound}
\norm{\OS_h^n}_{L^2(\Sigma) \to L^2(\Sigma)} \leq 4 M_m n \eta_h.
\end{equation}
In particular, if there exists $\delta\in [0,1)$ independent of $h$ and $n$
such that
\begin{equation}\label{eq:assumption on eta_h}
  \eta_h \leq \frac{\delta}{4M_m n},
\end{equation}
the discrete inf-sup constant satisfies
\begin{equation}\label{eq:imroved asymptotic inf-sup}
  \widetilde{ \alpha}_{\text{is}}^{\OV} \geq \frac{1-\delta}{2n}.
\end{equation}
\end{theorem}
\proof{
Let $Z_h\in S_{0,}^1(\Sigma_h)$ be arbitrary but fixed. The representation
\eqref{eq:S_h representation in S} together with $\norm{\OS}_{L^2(\Sigma)\to
L^2(\Sigma)}\leq 1$ gives
\begin{equation*}
\norm{\OS_h^n Z_h}_{L^2(\Sigma)}
=\norm{\sum_{j=0}^{n-1} \OS^{n-j-1} (\OI-\OQ_h^1)\OS \OS_h^j Z_h}_{L^2(\Sigma)}
\leq\sum_{j=0}^{n-1} \norm{(\OI-\OQ_h^1)\OS \OS_h^j Z_h}_{L^2(\Sigma)}.
\end{equation*}
Now for every $j=0,\ldots,n-1$ the function $V_{\whh} = \OS \OS_h^j Z_h$
is a piecewise affine linear function on the shifted mesh with nodes $\widehat{
\mathcal{N}}_h$. Using that the $L^2$-projection is the best approximation, the
nodal interpolation error estimate from
Lemma~\ref{lem:shifter meshes interpolation error estimate}, and $\norm{V_{\whh
}}_{L^2(\Sigma)}\leq \norm{Z_h}_{L^2(\Sigma)}$, we further get
\begin{equation*}
\norm{(\OI-\OQ_h^1)V_{\whh}}_{L^2(\Sigma)}
\leq \norm{(\OI-\mathcal{I}_h^1)V_{\whh}}_{L^2(\Sigma)}
\leq 4 M_m \eta_h \norm{V_{\whh}}_{L^2(\Sigma)}
\leq 4 M_m \eta_h \norm{Z_h}_{L^2(\Sigma)},
\end{equation*}
where $\mathcal{I}_h^1:\mathcal{C}([0,T])\to S_{h}^1(0,T)$ denotes the nodal
interpolation. This, after summation, gives
\eqref{eq:mesh and n dependent S_h bound}. In addition, if
\eqref{eq:assumption on eta_h} holds, using
Lemma~\ref{lem:asymptotic bound an alpha_is using norm S_h^n} gives
\eqref{eq:imroved asymptotic inf-sup}.
}

\begin{remark}
Condition~\eqref{eq:assumption on eta_h} is a relative mesh-alignment
condition. More precisely, every node of a mesh shifted by $L$ must satisfy
\[
  d_{i,\ast}
  \leq \frac{\delta}{4M_m n} \min\{\whh_{i,\ast},\whh_{i+1,\ast}\},
  \qquad \ast\in\{0,L\},
\]
i.e., the distance of a shifted node to the nearest node of the original mesh
on the opposite boundary must be at most of order $n^{-1}$ relative to the
local mesh size. Assumption~\ref{Assumption:mesh} corresponds to the limiting
case of exact alignment, for which $\eta_h=0$, but exact coincidence of the
shifted and target nodes is not required.
\end{remark}

We conclude by noticing that numerical experiments indicate that condition
\eqref{eq:assumption on eta_h} is sufficient but not necessary. Indeed, we refer
to Figure~\ref{fig:NumericalResults_BAmesh} for an example in which we have a
mesh that clearly violates condition~\eqref{eq:assumption on eta_h} yet still
leads to an inf-sup constant that behaves like $n^{-1}$.

\begin{figure}[H]
\centering
\begin{subfigure}{0.45\textwidth}
\begin{tikzpicture}

\shade[left color=white,right color=lightgray] (0,0) rectangle (3,3) node[pos=.5] {$T_1$};
\shade[left color=lightgray,right color=white] (0,3) rectangle (3,6) node[pos=.5] {$T_2$};

\draw[black] (0,0)--(0,6);
\draw[black] (3,0)--(3,6);

\def\leftTimes{ 0.00, 0.75, 0.7734375, 0.796875, 0.84375, 0.9375, 1.125,
 1.5, 4.5, 4.875, 5.0625, 5.15625, 5.203125, 5.2265625, 5.25, 6.0}
\def\rightTimes{0.00, 0.75, 0.7734375, 0.796875, 0.84375, 0.9375, 1.125,
 1.5, 4.5, 4.875, 5.0625, 5.15625, 5.203125, 5.2265625, 5.25, 6.0}

\foreach \t in \leftTimes {
    \draw (0,\t) node[circle, fill=black, scale=0.1, draw=black] {};
}

\foreach \t in \rightTimes {
    \draw (3,\t) node[circle, fill=black, scale=0.1, draw=black] {};
}

\draw[dotted] (-1.5,0)  -- (0,0) node[pos=-0.5] {$t=0$};
\draw[dotted] (-1.5,3)  -- (0,3) node[pos=-0.5] {$t=L$};
\draw[dotted] (-1.5,6)  -- (0,6) node[pos=-0.5] {$t=2L$};

\draw[dotted] (0,-1)  -- (0,0) node[pos=-0.3] {$\Sigma_h^0$};
\draw[dotted] (3,-1)  -- (3,0) node[pos=-0.3] {$\Sigma_h^L$};

\end{tikzpicture}
\caption{Example of badly aligned mesh for\\ $T=2L$, and $L=3$.}
\label{fig:badly_aligned_mesh}
\end{subfigure} \hfill
\begin{subfigure}{0.54\textwidth}
 \pgfplotstableread{Data/inf_sup_1stkind_dualH_Tdiff_BAmesh.txt}\infsupBA


{\scalefont{1.5}
\begin{tikzpicture}
\begin{loglogaxis}[
    log basis y ={2},
    log basis x = {2},
    width=14cm,
    height=14cm,
    xlabel={$n$},
    ylabel={inf-sup constant},
    xmin=1, xmax=10,
    ytick pos= left,
    xtick pos= bottom,
    xtick={1,2,3,4,5,6,7,8,9,10},
    xticklabels= {1,2,3,4,5,6,7,8,9,10},
    ytick = {0.5,0.25,0.125, 0.0625},
    legend pos=north east,
    legend cell align={left},
    legend style ={font = \large, inner sep= 6pt, outer sep = -6pt},
    ymajorgrids=true,
    grid style=dashed,
]


\addplot [
    domain= 1:10, 
    samples=10, 
    color=black,
    dashed,
    line width=2pt
]
{sin(deg(pi/(4*x+2)))};
\addlegendentry{$\alpha_{\text{is}}^{\OV}(n)$}




\addplot[orange,mark= square*,line width=2pt, mark options={solid}] table [x expr=\coordindex+1, y index = 0] {\infsupBA};
\addlegendentry{Inf-sup constant}

\end{loglogaxis}
\end{tikzpicture}
}%
\caption{Inf-sup constant for different final times $T=nL$ and a family of
meshes like the one in Figure~\ref{fig:badly_aligned_mesh}.}
\label{fig:infsup_BAmesh}
\end{subfigure}
\caption{Numerical results for \eqref{eq:Vpwc_pwl} over meshes that violate
condition~\eqref{eq:assumption on eta_h}.}
\label{fig:NumericalResults_BAmesh}
\end{figure}
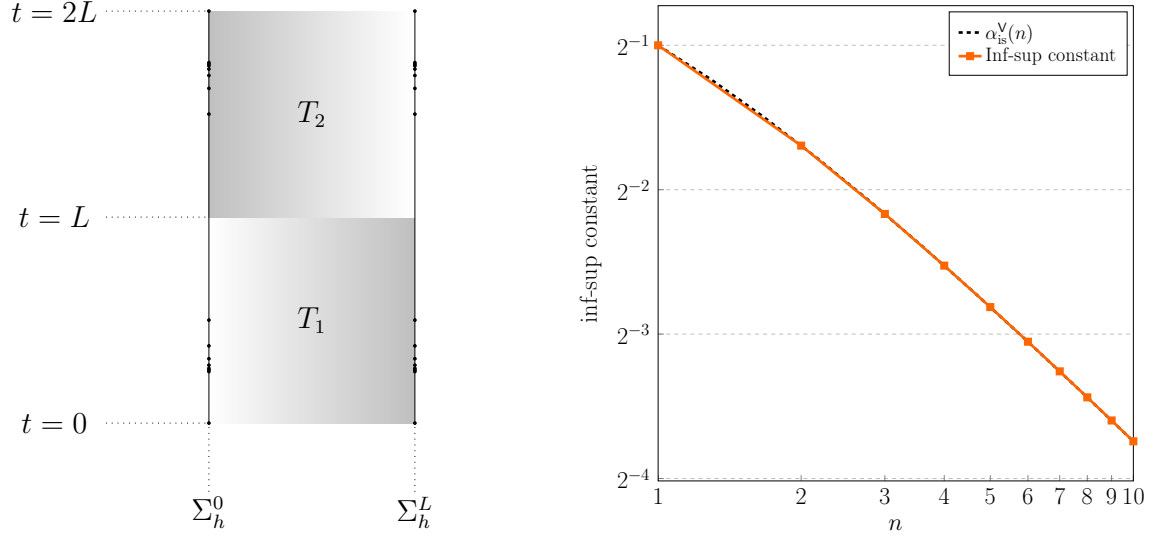

\paragraph{Declaration of AI use.} OpenAI’s ChatGPT 5.6-6.1 Sol was used during
the preparation of this manuscript to assist with computations and final
polishing (checking mathematical arguments, language revision, and typesetting).
The authors independently verified the mathematical content and take full
responsibility for the manuscript.



\appendix

\section{Auxiliary results for non-matching meshes}
In this section we derive technical interpolation estimates needed for
Theorem~\ref{thm:EstimatesWeakerMesh}, together with all the required building
blocks.

Let $0=t_0 < t_1 <\ldots < t_N = T$ be a decomposition of the interval $(0,T)$.
Moreover, let $S_h^1(0,T)$ be the space of piecewise linear, globally continuous
functions, defined on this decomposition.
\begin{lemma}\label{lem:representation of pw linears}
For $Z_h \in S_h^1(0,T)$ it holds that
\begin{equation*}
Z_h(t) = Z_h(0) + Z_h'(t_{1/2})t + \sum_{i=1}^{N-1} [Z_h']_i \max\{t-t_i,0\},
\end{equation*}
where $t_{i+1/2} = \tfrac{t_i+t_{i+1}}{2}$, $i=0,\ldots,N-1$, and
\begin{equation*}
  [Z_h']_i := Z_h'(t_{i+1/2}) - Z_h'(t_{i-1/2}), \quad i=1,\ldots,N-1.
\end{equation*}
\end{lemma}
\proof{
We define
\begin{equation*}
  v_h(t) = Z_h(0) + Z_h'(t_{1/2})t + \sum_{i=1}^{N-1} [Z_h']_i \max\{t-t_i,0\}.
\end{equation*}
Indeed, $v_h \in S_h^1(0,T)$, as it is affine linear on each element $(t_{i-1},
t_i).$ Now we show that actually $v_h = Z_h.$ First note that
\begin{align*}
  v_h(0) = Z_h(0) + \sum_{i=1}^{N-1}[Z_h']_i \max\{-t_i,0\} = Z_h(0),
\end{align*}
because $\max\{-t_i,0\} =0$ for all $i=1,\ldots,N$. Then, let $\mathbf{1}_{X}(t
)$ denote the characteristic function over $X\subset \R$.
Using that $\tfrac{d}{dt}\max\{t-t_i,0\} = \mathbf{1}_{t> t_i}(t)$,
we get for the derivative evaluated at $t\in (t_{j-1},t_j)$ that
\begin{align*}
v_h'(t) &= Z_h'(t_{1/2}) + \sum_{i=1}^{N-1} [Z_h']_i\mathbf{1}_{t>t_i}(t)\\
&= Z_h'(t_{1/2}) + \sum_{i=1}^{j-1}(Z_h'(t_{i+1/2})- Z_h'(t_{i-1/2}))
= Z_h'(t_{j-1/2}).
\end{align*}
Since $v_h$ is affine linear, we have that $v_h'(t)=v_h'(t_{j-1/2}) = Z_h'(t_{
j-1/2})$ for $t\in (t_{j-1},t_j)$, and thus we established the relations
$v_h(0)=Z_h(0)$ and $v_h' = Z_h'$ giving the desired equality.
}
\begin{lemma}\label{lem:error estimate max - I_h max}
Let $\mathcal{I}_h^1:\mathcal{C}([0,T])\to S_{h}^1(0,T)$ denote the nodal
interpolation of continuous functions. For any fixed $\widehat t\in (0,T)$, let
$m_{\widehat t}(t) := \max\{t-\widehat t,0\}$. Then, it holds that
\begin{equation}\label{eq:error estimate max - I_h max}
\norm{m_{\widehat t} - \mathcal{I}_h^1 m_{\widehat t}}_{L^2(0,T)}^2
= \begin{cases}
0, & \widehat t= t_i,\\
\frac{(\widehat t-t_{i-1})^2(t_i-\widehat t)^2}{3(t_i-t_{i-1})}, & \widehat
t\in (t_{i-1},t_i).
\end{cases}
\end{equation}
\end{lemma}
\proof{
Recall that $m_{\widehat t}(t) = \max\{t-\widehat t,0\}$ is a piecewise linear
globally continuous function. Hence, if $\widehat t = t_i$ for some $i=0,\ldots
,N$ then $m_{\widehat t} = \mathcal{I}_h^1 m_{\widehat t}$ and thus
\begin{equation*}
  \norm{m_{\widehat t} - \mathcal{I}_h^1 m_{\widehat t}}_{L^2(0,T)} =0.
\end{equation*}
So, let us assume that $\widehat t \in (t_{i-1},t_i)$, for $i=1,\ldots,N.$ Since
\begin{equation*}
m_{\widehat t}(t)
= \begin{cases}
  0, & t<t_{i-1},\\
  t-\widehat t, & t> t_i,
\end{cases}
\end{equation*}
we note that $ m_{\widehat t}(t)$ is interpolated exactly outside the interval
$(t_{i-1},t_i)$ and therefore $m_{\widehat t}(t) = (\mathcal{I}_h^1 m_{\widehat
t})(t)$ for $t\in (0,t_{i-1})\cup (t_i,T)$. In particular,
\begin{equation}\label{eq:suppoart of I-I_h^1 m_t}
  \text{supp}(m_{\widehat t}-\mathcal{I}_h^1 m_{\widehat t}) \subset [t_{i-1},
  t_i], \quad \text{if }\widehat t\in (t_{i-1},t_i).
\end{equation}
Moreover, we have that $m_{\widehat t}(t_{i-1}) = 0$ and $m_{\widehat t}(t_i) =
t_i-\widehat t$, and thus the interpolation is given by
\begin{equation*}
  (\mathcal{I}_h^1 m_{\widehat t})(t)
  = \frac{t-t_{i-1}}{t_i-t_{i-1}}(t_i-\widehat t), \quad  t \in (t_{i-1},t_i).
\end{equation*}
From this we can compute
\begin{equation*}
m_{\widehat t}(t) - (\mathcal{I}_h^1 m_{\widehat t})(t)
= \begin{cases}
-\frac{t-t_{i-1}}{t_i-t_{i-1}}(t_i-\widehat t\,), &t\in (t_{i-1},\widehat t\,],\\
\frac{\widehat t-t_{i-1}}{t_i-t_{i-1}}(t-t_i), & t\in (\widehat t,t_i).
\end{cases}
\end{equation*}
Then, integrating gives
\begin{align*}
\norm{m_{\widehat t} -\mathcal{I}_h^1 m_{\widehat t}}_{L^2(0,T)}^2
&=\norm{m_{\widehat t} -\mathcal{I}_h^1 m_{\widehat t}}_{L^2(t_{i-1},t_i)}^2\\
&=\int_{t_{i-1}}^{\widehat t} \left[\frac{t-t_{i-1}}{t_i-t_{i-1}}(t_i-\widehat
t\,)\right]^2\, dt
+\int_{\widehat t}^{t_i} \left[\frac{\widehat t-t_{i-1}}{t_i-t_{i-1}}(t-t_i)
\right]^2\, dt \\
&=\frac{(\widehat t-t_{i-1})^2(t_i-\widehat t\,)^2}{3(t_i-t_{i-1})}.
\end{align*}
}

Now, we proceed to consider the definitions from Section~\ref{sssec:appComp},
but written for the current setting. Let $0=\widehat t_0<\widehat t_1<\cdots<
\widehat t_{\widehat N}=T$ be another partition of the interval $(0,T)$, with
mesh sizes $\whh_i:=\widehat t_i-\widehat t_{i-1},$ $i=1,\ldots,\widehat
N,$ and let $S_{\whh}^1(0,T)$ denote the space of piecewise linear,
globally continuous functions with respect to this decomposition.
For every element $(t_{\ell-1},t_\ell)$ of the original partition, define
\[
J_\ell:=
\left\{i\in\{1,\ldots,\widehat N-1\}:\widehat t_i\in(t_{\ell-1},t_\ell)\right\},
\qquad \ell=1,\ldots,N.
\]
Thus, $J_\ell$ contains the indices of all interior nodes of the second
partition that lie strictly inside the element $(t_{\ell-1},t_\ell)$. We
assume that there exists a constant $M_m\in\mathbb N$ such that
\begin{equation}\label{eq:assumption on Ji}
  |J_\ell|\leq M_m, \qquad \ell=1,\ldots,N,
\end{equation}
and, for $h_\ell := t_\ell-t_{\ell-1}$, $\ell=1,\ldots,N$, we assume that
\begin{equation}\label{eq:assumption on h_ell}
  h_\ell \leq M_m \min\{\whh_i,\whh_{i+1}\},
  \quad \text{ for all } i\in J_\ell.
\end{equation}
In other words, every element of the original partition contains at most $M_m$
interior nodes of the second partition. Moreover, define the distance of a node
$\widehat t_i$ to the original partition by $d_i := \min_{\ell=0,\ldots,N}|
\widehat t_i - t_\ell|,$ and the quantity $\eta_h := \max_{i=1,\ldots,\widehat
N-1} \frac{d_i}{\min\{\whh_i, \whh_{i+1}\}}$. Then we have the
following result.
\begin{lemma}\label{lem:shifter meshes interpolation error estimate}
For partitions such that Assumption~\eqref{eq:assumption on h_ell} holds,
we have that for any $V_{\whh} \in S_{\whh}^1(0,T)$, the following
bound is satisfied
\begin{equation*}
\norm{(I-\mathcal{I}_h^1)V_{\whh}}_{L^2(0,T)}
\leq 4 M_m \eta_h \norm{V_{\whh}}_{L^2(0,T)}.
\end{equation*}
\end{lemma}
\proof{
By Lemma \ref{lem:representation of pw linears}, there holds the representation
\begin{equation*}
V_{\whh}(t) = V_{\whh}(0) + V_{\whh}'(0)t + \sum_{i=1}^{\hat{N} -1}
[V_{\whh}']_i m_{\widehat{t}_i}(t).
\end{equation*}
Since interpolation preserves linear functions, we get
\begin{equation*}
(\OI-\mathcal{I}_h^1) V_{\whh}(t)= \sum_{i=1}^{\widehat N -1}[V_{\whh}']_i
(\OI-\mathcal{I}_{h}^1)m_{\widehat t_i}(t).
\end{equation*}
Furthermore, for $\widehat t_i \in (t_{\ell-1},t_\ell)$ it holds that
$\text{supp}((\OI-\mathcal{I}_h^1)m_{\widehat t_i})\subset [t_{\ell-1},t_\ell]$,
see \eqref{eq:suppoart of I-I_h^1 m_t}, and therefore we derive
\begin{align}\label{eq:boundHappy}
\norm{(\OI-\mathcal{I}_h^1)V_{\whh}}_{L^2(0,T)}^2
&= \sum_{\ell=1}^{N} \norm{(\OI-\mathcal{I}_h^1)V_{\whh}}_{L^2(t_{\ell-1},
  t_\ell)}^2
=\sum_{\ell=1}^{N} \norm{\sum_{i=1}^{\widehat N-1}[V_{\whh}']_i(\OI-
\mathcal{I}_h^1)m_{\widehat t_i}}_{L^2(t_{\ell-1},t_\ell)}^2 \nonumber \\
&= \sum_{\ell=1}^{N} \norm{\sum_{i\in J_\ell}[V_{\whh}']_i(\OI-\mathcal{I}_h^1)
m_{\widehat t_i}}_{L^2(t_{\ell-1},t_\ell)}^2.
\end{align}
Moreover, the Cauchy--Schwarz inequality gives us that
\begin{align*}
\norm{\sum_{i\in J_\ell}[V_{\whh}']_i(\OI-\mathcal{I}_h^1)m_{\widehat t_i}}_{
L^2(t_{\ell-1},t_\ell)}^2  & \quad \leq
\int_{t_{\ell-1}}^{t_\ell} \left(\sum_{i\in J_\ell} 1\right) \left(\sum_{i\in
J_\ell} [V_{\whh}']_i^2 ((\OI-\mathcal{I}_h^1)m_{\widehat t_i})(t)^2\right)\,dt\\
&\quad =|J_\ell|\left(\sum_{i\in J_\ell} [V_{\whh}']_i^2 \norm{(\OI-
\mathcal{I}_h^1)m_{\widehat t_i}}_{L^2(t_{\ell-1},t_\ell)}^2\right).
\end{align*}
Now, by assumption \eqref{eq:assumption on Ji} and using the
error estimate \eqref{eq:error estimate max - I_h max} together with
\begin{equation*}
\frac{(\widehat t_i -t_{\ell-1})^2(t_\ell - \widehat t_i)^2}{3h_\ell}
= \frac{\min\{(\widehat t_i-t_{\ell-1})^2, (t_\ell-\widehat t_i)^2\}
\max\{(\widehat t_i-t_{\ell-1})^2, (t_\ell-\widehat t_i)^2\}}{3h_\ell}
\leq d_i^2 \frac{h_\ell}{3},
\end{equation*}
we can further estimate
\begin{align*}
|J_\ell|\left(\sum_{i\in J_\ell} [V_{\whh}']_i^2 \norm{(\OI-\mathcal{I}_h^1)
m_{\widehat t_i}}_{L^2(t_{\ell-1},t_\ell)}^2\right)
&\quad\leq M_m \sum_{i\in J_\ell} [V_{\whh}']_i^2\frac{(\widehat t_i -t_{\ell-1})^2
(t_\ell-\widehat t_i)^2}{3(t_\ell-t_{\ell-1})}\\
&\quad \leq M_m \sum_{i\in J_\ell} [V_{\whh}']_i^2 d_i^2\frac{h_\ell}{3}.
\end{align*}
Additionally, using \eqref{eq:assumption on h_ell}, we get
\begin{align*}
d_i^2 h_\ell & =\frac{d_i^2}{\min\{\whh_i,\whh_{i+1}\}^2} \min\{\whh_i,
\whh_{i+1}\}^2 h_\ell \leq M_m\eta_h^2 \min\{\whh_i,\whh_{i+1}\}^3,
\end{align*}
which, using that $J_\ell$ are pairwise disjoint, implies the following
intermediate estimate
\begin{align*}
\norm{(I-\mathcal{I}_h^1)V_{\whh}}_{L^2(0,T)}^2
&\leq \frac{M_m^2\eta_h^2}{3}\sum_{\ell=1}^{N}\sum_{i\in J_\ell} \min\{\whh_i,
\whh_{i+1}\}^3 [V_{\whh}']_i^2\\
&\leq \frac{M_m^2\eta_h^2}{3}\sum_{i=1}^{\widehat N-1} \min\{\whh_i,\whh_{i+1}
\}^3 [V_{\whh}']_i^2.
\end{align*}
Furthermore, we bound
\begin{equation*}
[V_{\whh}']_i^2
=(V_{\whh}'(\widehat t_{i+1/2}) - V_{\whh}'(\widehat t_{i-1/2}))^2
\leq 2\big(V_{\whh}'(\widehat t_{i+1/2})^2 + V_{\whh}'(\widehat t_{i-1/2}
)^2\big).
\end{equation*}
As $V_{\whh}$ is piecewise linear, $V_{\whh}'(t) = V_{\whh}'(
\widehat t_{i-1/2})$ for $t\in (\widehat t_{i-1},\widehat t_i)$, $i=1,\ldots,
\widehat N-1$. Thus,
\begin{align*}
\min\{\whh_i,\whh_{i+1}\}^3[V_{\whh}']_i^2
&\leq 2\min\{\whh_i,\whh_{i+1}\}^3(V_{\whh}'(\widehat t_{i-1/2})^2
+ V_{\whh}'(\widehat t_{i+1/2})^2)\\
&\leq 2\whh_i^3 V_{\whh}'(\widehat t_{i-1/2})^2 + 2\whh_{i+1}^3
V_{\whh}'(\widehat t_{i+1/2})^2\\
&= 2\left(\whh_i^2\int_{\widehat t_{i-1}}^{\widehat t_i}V_{\whh}'(t)^2\,dt
+\whh_{i+1}^2\int_{\widehat t_{i}}^{\widehat t_{i+1}}
V_{\whh}'(t)^2 \, dt \right).
\end{align*}
Now a sharp inverse inequality
\begin{equation*}
\whh_{i}^2\norm{V_{\whh}'}_{L^2(\widehat t_{i-1},\widehat t_i)}^2
\leq 12 \norm{V_{\whh}}_{L^2(\widehat t_{i-1},\widehat t_i)}^2,
\quad i=1,\ldots,\widehat N,
\end{equation*}
gives the estimate
\begin{align*}
\sum_{i=1}^{\widehat N-1}\min\{\whh_i,\whh_{i+1}\}^3[V_{\whh}']_i^2
&\leq 2\sum_{i=1}^{\widehat N-1} \big(\whh_i^2 \norm{V_{\whh}'}_{L^2(\widehat
t_{i-1},\widehat t_i)}^2 + \whh_{i+1}^2 \norm{V_{\whh}'}_{L^2(\widehat t_{i},
\widehat t_{i+1})}^2\big) \\
&\leq 4\sum_{i=1}^{\widehat N} \whh_i^2 \norm{V_{\whh}'}_{L^2(\widehat t_{i-1},
\widehat t_i)}^2
\leq 48 \norm{V_{\whh}}_{L^2(0,T)}^2.
\end{align*}
And we obtain the final estimate
\begin{equation*}
\norm{(I-\mathcal{I}_h^1)V_{\whh}}_{L^2(0,T)}^2
\leq 16 M_m^2 \eta_h^2\norm{V_{\whh}}_{L^2(0,T)}^2,
\end{equation*}
and taking the square root finishes the proof.
}

\end{document}


\pgfplotstableread{Data/inf_sup_1stkind_dualH_Tdiff_uniform_LS.txt}\infsupUni


{\scalefont{1.5}
\begin{tikzpicture}
\begin{loglogaxis}[
    log basis y ={2},
    log basis x = {2},
    width=14cm,
    height=14cm,
    xlabel={$n$},
    ylabel={inf-sup constant},
    xmin=1, xmax=10,
    ytick pos= left,
    xtick pos= bottom,
    xtick={1,2,3,4,5,6,7,8,9,10},
    xticklabels= {1,2,3,4,5,6,7,8,9,10},
    ytick = {0.5,0.25,0.125, 0.0625},
    legend pos=north east,
    legend cell align={left},
    legend style ={font = \large, inner sep= 6pt, outer sep = -6pt},
    ymajorgrids=true,
    grid style=dashed,
]


\addplot [
    domain= 1:10, 
    samples=10, 
    color=black,
    dashed,
    line width=2pt
]
{sin(deg(pi/(4*x+2)))};
\addlegendentry{$\alpha_{\text{is}}^{\OV}(n)$}

\addplot [
    domain= 1:10, 
    samples=10, 
    color=red,
    dotted,
    line width=2pt
]
{sqrt(3/4)*sin(deg(pi/(4*x+2)))};
\addlegendentry{$\sqrt{3/4}\alpha_{\text{is}}^{\OV}(n)$}



\addplot[orange,mark= square*,line width=2pt, mark options={solid}] table [x expr=\coordindex+1, y index = 0] {\infsupUni};
\addlegendentry{Inf-sup constant - Uniform}


\end{loglogaxis}
\end{tikzpicture}
}